\documentclass[11pt, a4paper]{amsart}

\RequirePackage[l2tabu,orthodox]{nag}
\usepackage{array}
\usepackage[T1]{fontenc}
\usepackage[utf8]{inputenc}
\usepackage{lmodern}
\usepackage[english]{babel}
\usepackage{color} 
\usepackage[usenames,dvipsnames]{xcolor}
\usepackage[alphabetic]{amsrefs}
\usepackage{amsthm,amssymb,amsmath}
\usepackage[pdftex]{graphicx}
\usepackage{enumerate}
\usepackage[inline]{enumitem}
\usepackage{url}
\usepackage[margin=2.5cm]{geometry}

\usepackage{tikz-cd}
\usepackage{mathrsfs}
\usepackage{xcolor}
\usepackage{MnSymbol} 

\usepackage{mathtools}
\numberwithin{equation}{section}

\usepackage{subfig, caption} 
\usepackage{wrapfig}
\usepackage{pgfplots}
\pgfplotsset{width=10cm,compat=1.9}

\usepackage{indentfirst}
\usepackage{microtype}

\usepackage[colorlinks=true,linkcolor=blue,citecolor=magenta]{hyperref}
\usepackage{comment}

\input{header_labeling.tex}

\newcommand{\NN}{\mathbb{N}}

\newcommand{\RR}{\mathbb{R}}

\newcommand{\ZZ}{\mathbb{Z}}

\newcommand{\cC}{\mathcal{C}}

\newcommand{\cS}{\mathcal{S}}

\newcommand{\Homeo}{\operatorname{Homeo}}
\newcommand{\Homeop}{\operatorname{Homeo_+}}

\newcommand{\Fix}{\operatorname{Fix}}

\newcommand{\Stab}[2]{\operatorname{Stab}_{#1}(#2)}
\newcommand{\Gap}{\operatorname{Gap}}
\newcommand{\Core}{\operatorname{Core}}

\newcommand{\PSL}{\operatorname{PSL}_2}

\newcommand{\SO}{\operatorname{SO}(2)}

\newcommand{\id}{\mathrm{id}}

\newcommand{\opi}[3][S^1]{(#2, #3)_{#1}}
\newcommand{\ropi}[3][S^1]{[#2, #3)_{#1}}
\newcommand{\lopi}[3][S^1]{(#2, #3]_{#1}}
\newcommand{\cldi}[3][S^1]{[#2, #3]_{#1}}

\newcommand{\closure}[1]{\overline{#1}}

\usepackage{array}
\newcolumntype{L}{>{\centering\arraybackslash}m{4.3cm}}

\title{Ping-pong dynamics of hyperbolic-like actions with non-simple points}

\author[Kim]{KyeongRo Kim}
\address{\hskip-\parindent
Research institute of Mathematics\\
Seoul National University\\
GwanAk-Ro 1, GwanAk-Gu, Seoul 08826, Korea}
\email{kyeongrokim14@gmail.com}

\author[Triestino]{Michele Triestino}
\address{\hskip-\parindent
Michele Triestino\\
	Université Bourgogne Europe, CNRS, IMB UMR 5584, 21000 Dijon, France\\
    \& Institut Universitaire de France}
\email{michele.triestino@ube.fr}

\date{\today}

\begin{document}

\begin{abstract} 
A hyperbolic-like group is a subgroup of $\Homeop(S^1)$ such that every non-trivial element has exactly two fixed points, one attracting and one repelling.
We investigate the ping-pong dynamics of hyperbolic-like groups, inspired by a conjecture of Bonatti.
We show the existence of a proper ping-pong partition for any pair of non-cyclic point stabilizers.
More precisely, our results explicitly provide such a ping-pong partition.
\end{abstract}

 
\maketitle
\section{Introduction} \label{Sec:intro}

An element $g$ in the group $\Homeop(S^1)$ of orientation-preserving circle homeomorphisms is \emph{hyperbolic-like} if $g$ has exactly two fixed points such that one of them is attracting and the other repelling. 
We say that a non-trivial subgroup $G\le \Homeop(S^1)$ is \emph{hyperbolic-like} if any non-trivial element is hyperbolic-like. This property is invariant under conjugacy in $\Homeop(S^1)$. A (faithful) group action $G\to\Homeo_+(S^1)$ is \emph{hyperbolic-like} if its image is.

The main examples of hyperbolic-like (sub)groups are given by torsion-free discrete subgroups of $\PSL(\RR)$ (and their conjugates in $\Homeop(S^1)$), but there are examples of different nature, first exhibited by Kova\v{c}evi\'{c} \cite{Kovacevic2} in relation to the celebrated works of Casson and Jungreis \cite{CassonJungreis} and Gabai \cite{Gabai} on convergence groups. The examples of Kova\v{c}evi\'{c} are built by operating a kind of surgery of subgroups of $\PSL(\RR)$, recently formalized by Carnevale in his Ph.D.\ thesis \cite{theseJoao}.  It is expected that this is the only way to produce examples of hyperbolic-like subgroups, up to \emph{semi-conjugacy}; see \cite{theseJoao}*{Conjecture 1.12}.
\begin{conj}[Bonatti]\label{Conj:Bonatti}
    Let $G\leq \Homeop(S^1)$ be a subgroup such that any non-trivial element has at most $2$ fixed points, and whose action is minimal. Assume that the action of $G$ is not semi-conjugate to any subgroup of $\PSL(\RR)\leq  \Homeop(S^1)$. Then, $G$ splits as an amalgamated product over an abelian subgroup.
\end{conj}
This problem is related to a conjecture of Frankel \cites{Frankel13,Frankel18}: 
\begin{conj}[Frankel]\label{Conj:Frankel}
Closed hyperbolic three-manifold groups do not admit hyperbolic-like actions on the circle.
\end{conj}

When a hyperbolic-like group $H$ has a global fixed point, classical results of Hölder \cite{Holder} and Solodov \cite{Solodov} (or a relatively modern result of Kova\v{c}evi\'{c} \cite{Kovacevic1}) give that the group is isomorphic to a subgroup of $(\RR,+)$, and its action is semi-conjugate to the action of an elementary loxodromic subgroup of $\PSL(\RR)$; see \refthm{elementary} for the precise statement. We will say that a hyperbolic-like group is \emph{non-elementary} if it does not admit any global fixed point; equivalently, a hyperbolic-like group is non-elementary if and only if it is non-abelian (see \refcor{commute}). After the work \cite{BonattiCarnevaleTriestino24} from the thesis of Carnevale, we have that a non-elementary hyperbolic-like subgroup of $\Homeop(S^1)$ is necessarily \emph{(locally) discrete}; see  \refprop{locDiscrete}.

In this paper, we consider the case of non-elementary hyperbolic-like groups containing non-cyclic point-stabilizers (from the discussion above, these are non-cyclic free abelian), and show that there is a strong dynamical constraint on such subgroups. Instead of giving here the precise statements, we highlight a straightforward consequence of the main results, in accordance with the prediction of \refconj{Bonatti}.

\begin{cor}
    Let $G$ be a non-elementary hyperbolic-like group. If $G$ is generated by a pair of non-cyclic abelian subgroups $H$ and $K$, then $G$ is isomorphic to the free product of $H$ and $K$.
\end{cor}

More precisely, we prove that the subgroups $H$ and $K$ from the statement admit an explicit ping-pong partition, and thus $\langle H,K\rangle \cong H*K$, by the classical ping-pong lemma (see for instance Maskit \cite{Maskit}*{Theorem VII.A.10}). Here we say that given two subgroups $H$ and $K$ of $\Homeop(S^1)$, a \emph{proper ping-pong partition} $(U_H,U_K)$ for $H$ and $K$ is a pair of non-empty disjoint open subsets $U_H$ and $U_K$ of $S^1$ with finitely many connected components such that
\[(H\setminus \{\id\})(U_K)\subset U_H\quad \text{and}\quad (K\setminus \{\id\})(U_H)\subset U_K.\]
Our description of the possible ping-pong partitions depends on the configuration of the fixed points of the subgroups $H$ and $K$, and it will postponed to the next section, after the introduction of some preliminary terminology and notation.

\section{Statement of the main results}

\subsection{Basic notation}

Here we introduce the main notation that we will use throughout the text.

\subsubsection{Topology of $S^1$}
In this paper we consider the usual \emph{circular order} $c$ on the circle $S^1$ defined as follows:
given a triple $(x,y,z)\in S^1\times S^1\times S^1$, we set $c(x,y,z)=1$ if $x,y,z$ are arranged in counterclockwise order, $c(x,y,z)=-1$ if $x,y,z$ are arranged in clockwise order and $c(x,y,z)=0$, otherwise.
Given $n\ge 3$ points $x_1,\ldots, x_n\in S^1$, we write $x_1<\cdots<x_n$ if $c(x_1,x_i,x_{i+1})=1$ for all $i\in \{2,\ldots,n-1\}$; when $c(x_1,x_i,x_{i+1})\ge 0$ for all $i\in \{2,\ldots,n-1\}$, we write $x_1\leq  \cdots\leq x_n$.
For a pair of points $x,y\in S^1$, we define the open interval $\opi{x}{y}$ as
$\opi{x}{y}=\{z\in S^1: c(x,z,y)=1\}$.
In the usual way, we also define the intervals $\lopi{x}{y}$, $\ropi{x}{y}$, and $\cldi{x}{y}$.
For any two-point subsets $\{x,y\}$ and $\{u,v\}$ of $S^1$, we say that $\{x,y\}$ and $\{u,v\}$ are \emph{linked} if $|\opi{x}{y}\cap \{u,v\}|=|\opi{y}{x}\cap \{u,v\}|=1$. Otherwise, they are \emph{unlinked}.

For a closed subset $A$ of $S^1$ containing at least two points, a component of $S^1\setminus A$ is called a \emph{gap} of $A$.
We say that a point $p$  in $A$ is \emph{accumulated on one side} in $A$ if there is a unique gap of $A$ whose boundary contains $p$. If there is no gap of $A$ whose boundary contains $p$, then $p$ is said to be \emph{accumulated on two sides} in $A$.

\subsubsection{Dynamics on one-manifolds}
When $G\le \Homeop(S^1)$ and $p\in X$ a point, we write $\Stab{G}{p}$ for the stabilizer of $p$ in $G$, and $\Fix(G)=\bigcap_{g\in G}\Fix(g)$ for the collection of fixed points of $G$.

Now, let $X$ be $\RR$ or $S^1$. Recall that $\RR$ has a standard linear order $<$, and $S^1$ is equipped with the circular order $c$. We denote the order-preserving homeomorphism group of $X$ by $\Homeop(X)$. A \emph{monotone map} $m$ on $X$ is a continuous surjective map satisfying the following:
\begin{itemize}
    \item the preimage of any point is connected,
    \item if $X=\RR$, $m$ is non-decreasing;
    \item if $X=S^1$, for any $(x,y,z)$ with $c(x,y,z)=1$, one has
    $c(m(x),m(y),m(z))\geq 0$. 
\end{itemize}
Let $G$ and $H$ be subgroups of $\Homeop(X)$. We say that $G$ is \emph{semi-conjugate} to $H$ if there is a monotone map $m\colon X\to X$ and a surjective homomorphism $\theta\colon G\to H$ such that $\theta(g)(x)=m(g(x))$ for any $g\in G$ and $x\in X$. Note that semi-conjugacy is not an equivalence relation.
For $m$ as above, we denote by $\Gap(m) \subset X$ the open subset of points at which $m$ is locally constant. We call each component of $\Gap(m)$ a \emph{gap} for $m$. Note that $m$ is a homeomorphism if and only if $\Gap(m)=\emptyset$. Finally, we define the \emph{core} of $m$ as the complement $\Core(m)=X \setminus \Gap(m)$.

\subsection{Results}

Let $G$ be a hyperbolic-like group. As mentioned in \refsec{intro}, point-stabilizers for the $G$-action are elementary subgroups, and hence isomorphic to subgroups of $(\RR,+)$. We say that a point $p$ in $S^1$ is \emph{non-simple} (for $G$) if its stabilizer is neither trivial nor infinite cyclic. Note that when $\Stab{G}{p}$, there exists a unique point $\overline p\in S^1\setminus \{p\}$ such that $\Stab{G}{p}=\Stab{G}{\overline p}$; see \refthm{elementary}. We call $\overline{p}$ the \emph{companion point} of $p$.
We will consider hyperbolic-like groups for which there are non-simple points $p,q\in S^1$ such that $p\notin \{q,\overline{q}\}$. Under this assumption, such a group is necessarily non-elementary (see \refcor{commute}). We will say that such non-simple points $p$ and $q$ are \emph{linked} if the pairs $\{p,\overline{p}\}$ and $\{q,\overline{q}\}$ are linked, and \emph{unlinked} otherwise.

Our main results provide explicit proper ping-pong partitions for pairs of stabilizers of non-simple points in a hyperbolic-like groups. In the linked case, we see that the dynamics can be described by the most natural ping-pong partition.

\begin{thmA}[linked case]\label{Thm1:linked}
Let $G$ be a hyperbolic-like group. If $p$ and $q$ are linked non-simple points, then there are disjoint open intervals $J_p$, $J_{\overline{p}}$, $J_q$ and $J_{\overline{q}}$ in $S^1$ with $\alpha\in J_\alpha$ for $\alpha\in \{p, \overline{p}, q, \overline{q}\}$, and such that $(J_p\cup J_{\overline p},J_q\cup J_{\overline q})$ is a proper ping-pong partition for $\Stab{G}{p}$ and $\Stab{G}{q}$:
\[(\Stab{G}{p}\setminus \{
\id\})(J_{q}\cup J_{\overline{q}})\subset J_p\cup J_{\overline{p}}\quad \text{and}\quad (\Stab{G}{q}\setminus \{
\id\})(J_{p}\cup J_{\overline{p}})\subset J_q\cup J_{\overline{q}}.\]
See \reffig{possiblePingPong} (left).
\end{thmA}

This result is proved in \refsec{linked}; we refer to \refcor{pingPongLinked} for a more precise statement. The unlinked case is discussed in \refsec{unlinked}. As a summary of the results of this section, we find three kinds of ping-pong partitions that can describe the dynamics.

\begin{thmA}[unlinked case]\label{Thm1:unlinked}
Let $G$ be a hyperbolic-like group. 
If $p$ and $q$ are unlinked non-simple points, then one of the following cases holds.
\begin{enumerate}
    \item{\emph{(geometric case)}} There are two disjoint intervals $U_p$ and $U_q$ in $S^1$ with $\{p,\overline{p}\}\subset U_p$, $\{q,\overline{q}\}\subset U_q$, and such that $(U_p,U_q)$ is a proper ping-pong partition for $\Stab{G}{p}$ and $\Stab{G}{q}$:
    \[(\Stab{G}{p}\setminus \{
\id\})(U_q)\subset U_p\quad\text{and} \quad(\Stab{G}{q}\setminus \{
\id\})(U_p)\subset U_q.\]
See \reffig{possiblePingPong} (center).
    
\item{\emph{(non-geometric case 1)}}\label{Itm:nonGeom1} There are disjoint open intervals $U$, $V$, $J_q$ and $J_{\overline{q}}$ in $S^1$ with $\{p,\overline{p}\}\subset U$, $\alpha\in J_\alpha$ for $\alpha\in \{q,\overline{q}\}$ and such that $(U\cup V,J_q\cup J_{\overline q})$ is a proper ping-pong partition for $\Stab{G}{p}$ and $\Stab{G}{q}$:
\[(\Stab{G}{p}\setminus \{
\id\})(J_q \cup J_{\overline {q}})\subset U\cup V \quad\text{and}\quad(\Stab{G}{q}\setminus \{
\id\})(U\cup V)\subset J_q\cup J_{\overline{q}}.\]
See \reffig{possiblePingPong} (right).
\item \emph{(non-geometric case 2)}\label{Itm:nonGeom2} There are disjoint open intervals $U$, $V$, $J_p$ and $J_{\overline{p}}$ in $S^1$ with $\{q,\overline{q}\}\subset U$, $\alpha\in J_\alpha$ for $\alpha\in \{p,\overline{p}\}$ and such that $(J_p\cup J_{\overline p},U\cup V)$ is a proper ping-pong partition for $\Stab{G}{p}$ and $\Stab{G}{q}$:
\[(\Stab{G}{p}\setminus \{
\id\})(U \cup V)\subset J_p\cup J_{\overline p} \quad\text{and}\quad(\Stab{G}{q}\setminus \{
\id\})(J_p\cup J_{\overline p})\subset U\cup V.\] 
\end{enumerate}
Moreover, if $p$ and $q$ lie in the same orbit, the non-geometric cases cannot occur.
\end{thmA}

The ping-pong partition in the geometric case is very natural, whereas we do not actually know whether the non-geometric case can occur.

\begin{figure}[ht]
        \centering
        \includegraphics[width=0.3\linewidth]{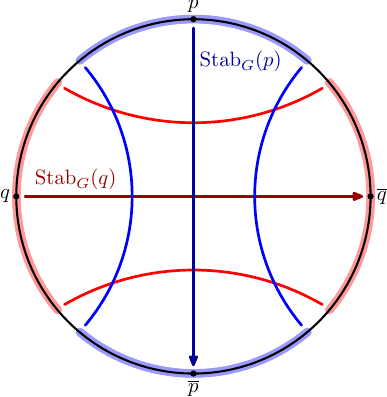}
        \includegraphics[width=0.3\linewidth]{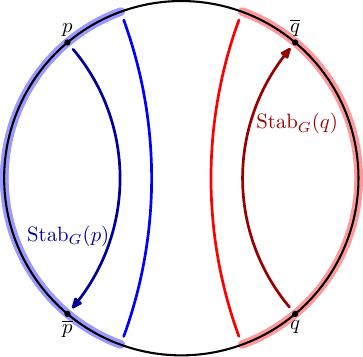}
\includegraphics[width=0.3\linewidth]{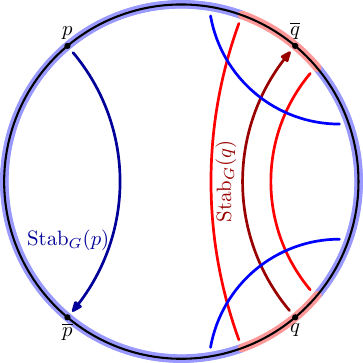}
        \caption{Possible configurations of gaps and associated ping-pong partitions for the linked case (left) and the unlinked cases. The red and blue chords represent gaps of $\Stab{G}{q}$ and of $\Stab{G}{p}$, respectively. The highlighted intervals describe the associated ping-pong partitions.}
        \label{Fig:possiblePingPong}
\end{figure}

\subsection{Overall strategy}
The first important observation, which is a consequence of \cite{BonattiCarnevaleTriestino24}, is that the stabilizer of a non-simple point $p$ cannot act minimally on both connected components of $S^1\setminus\{p,\overline p\}$ (\refprop{gapFromLimitSet}). The wandering intervals in these components are called the \emph{gaps} of $p$. To prove of our main results, we analyze the possible configurations of gaps of given non-simple points $p$ and $q$ (with $q\notin\{p,\overline p\}$) to find a ping-pong partition. In a hyperbolic-like group $G$, the fixed points of compositions and commutators of elements in the stabilizers $\Stab{G}{q}$ and $\Stab{G}{p}$ need to satisfy many combinatorial restrictions (these quite technical results are discussed in the appendix), and this is used to get  restrictions on the possible arrangements of gaps of $p$ and $q$. The analysis for the linked case is used in the discussion of the unlinked case in order to reduce the number of cases to consider.

\section{Basic dynamics of hyperbolic-like groups}
\subsection{Elementary hyperbolic-like groups}

A subgroup $H$ of $\Homeop(S^1)$ is \emph{elementary} if there is an $H$-invariant Borel probability measure on $S^1$. Equivalently, $H$ is elementary if it either admits a finite orbit, or is semi-conjugate to a subgroup of $\SO$ (see Ghys \cite{Ghys01}*{Proposition 6.17}). For hyperbolic-like groups, the situation is much more restrictive: any elementary hyperbolic-like group is abelian, and one can give a precise description of the possible dynamics, relying on a classical result of Hölder. The following statement can be deduced from the pioneering work of Kova\v{c}evi\'{c} \cite{Kovacevic1}, and we refer to the work of Bonatti, Carnevale, and the second author \cite{BonattiCarnevaleTriestino24}*{proof of Theorem A} for a proof in a more modern language.

\begin{thm}\label{Thm:elementary}
Any elementary hyperbolic-like group $H$ fixes exactly two points: there exist two distinct points $p,\overline{p}\in S^1$ such that $\Fix(H)=\{p,\overline p\}$. In addition, the following holds: write $R(p)=\opi{p}{\overline p}$ and $L(p)=\opi{\overline p}{p}$, and choose the linear orders on both intervals such that the directions from $p$ to $\overline p$ are the increasing directions; then for any $I\in \{R(p),L(p)\}$, there exist a monotone map
$m_{p,I}\colon I\to \RR$, and an injective homomorphism
$\theta_{p,I}\colon H\to (\RR,+)$, unique up to positive rescaling,
such that
\[
m_{p,I}(g(x))=m_{p,I}(x)+\theta_{p,I}(g)\quad \text{for any $x\in I$ and $g\in H$}.
\]
Moreover, the morphisms $\theta_{p,R(p)}$ and $\theta_{p,L(p)}$ are positively proportional.
\end{thm}

Let us point out the following straightforward consequence.

\begin{cor}\label{Cor:commute}
Let $G$ be a hyperbolic-like group. For any $f,g$ in $G\setminus \{\id\}$, the following conditions are equivalent:
\begin{enumerate}[label=(\arabic*)]
    \item the subgroup generated by $f$ and $g$ is elementary,
    \item $\Fix(f)\cap \Fix(g)\neq \emptyset$,
    \item $\Fix(f)=\Fix(g)$,
    \item $f$ and $g$ commute.
\end{enumerate}
\end{cor}

\subsection{Non-elementary hyperbolic-like groups}

When a subgroup $G$ of $\Homeop(S^1)$ is non-elementary, there exists a unique non-empty closed $G$-invariant subset $\Lambda(G)$ of $S^1$ that is minimal with respect to these properties, and it is either the whole circle $S^1$ or a Cantor set (see Ghys \cite{Ghys01}*{Proposition~5.6}). Here, we refer to $\Lambda(G)$ as the \emph{limit set} of $G$.
We say that a non-empty open interval $I$ in $S^1$ is \emph{wandering} (under the $G$-action) if $I\cap \Lambda(G)= \emptyset$. We also call each component of $S^1\setminus \Lambda(G)$ a \emph{gap} of $\Lambda(G)$. Obviously, a gap of $\Lambda(G)$ is wandering, and every wandering interval is contained in some gap.  Note that when $\Lambda(G)=S^1$, every non-empty open interval is non-wandering.

For hyperbolic-like groups, the following result characterizes the limit set as the closure of the set of fixed points.

\begin{lem}\label{Lem:limitSet}
Let $G$ be a non-elementary hyperbolic-like group. Then, $\Lambda(G)$ is the closure of the subset
\[\bigcup_{g\in G\setminus \{\id\}}\Fix(g).\]
Therefore, a non-empty open interval is non-wandering if and only if it contains a fixed point of some non-trivial element of $G$.
\end{lem}

\begin{proof}
   Set $\Lambda_0=\bigcup_{g\in G\setminus \{\id\}}\Fix(g)$. Let us first prove that $\Lambda_0\subset \Lambda(G)$. For this, assume by contradiction that there is a point $x\in  \Lambda_0\setminus\Lambda(G)$. Let $I=\opi{y}{z}\subset S^1$ be the gap of $\Lambda(G)$ that contains $x$, and take $g\in G\setminus \{\id\}$ such that $g(x)=x$; let us assume that the other fixed point of $g$ does not belong to the interval $\opi{y}{x}$ (otherwise, we argue similarly with the interval $\opi{x}{z}$), and that $x$ is an attracting fixed point for $g$ (otherwise, we replace $g$ with $g^{-1}$). If so, on the one hand we have $g(y)\in \opi{y}{x}$, and on the other hand we also have $g(y)\in g(\Lambda(G))=\Lambda(G)$, and this yields the desired contradiction. For the other inclusion, note that for any $h,g\in G$ we have the relation $h(\Fix(g))=\Fix(hgh^{-1})$; thus $\Lambda_0$ is $G$-invariant, and so is its closure $\overline{\Lambda_0}$. By minimality, this implies $\Lambda(G)\subset \overline{\Lambda_0}$.
\end{proof}

After the work of Bonatti, Carnevale, and the second author \cite{BonattiCarnevaleTriestino24}, we have that non-elementary hyperbolic-like groups are discrete in $\Homeop(S^1)$ in a very strong sense.

\begin{defn}
Let $G$ be a subgroup of $\Homeop(S^1)$.
Given a non-empty open interval $I \subset S^1$, we say that $G$ is \emph{discrete on $I$} if the identity $\id_I$ on $I$ is isolated among the collection of restrictions $\{g{\restriction_I} : g \in G\} \subseteq C^0(I,S^1)$ with respect to the $C^0$ topology.
When $G$ is non-elementary, we say that $G$ is \emph{locally discrete} if it is discrete on each non-wandering interval for $G$. 
\end{defn}

\begin{rmk}\label{rmk:local-nonlocal-equiv}
    In fact, for non-elementary subgroups $G\le \Homeop(S^1)$, discreteness and local discreteness are equivalent; see \cite{BonattiCarnevaleTriestino24}*{Lemma 4.5}.
\end{rmk}

\begin{prop}\label{Prop:locDiscrete}
Let $G\le \Homeop(S^1)$ be a non-elementary, non-locally discrete subgroup such that any non-trivial element has at most $N$ fixed points, for some uniform $N\ge 1$. Then $G$ contains elements without fixed points.
In particular, any non-elementary hyperbolic-like subgroup of $\Homeop(S^1)$ is discrete (equivalently, locally discrete).
\end{prop}

\begin{proof}
    The first statement is a consequence of \cite{BonattiCarnevaleTriestino24}*{Lemma 4.6} and the second is a straightforward application.
\end{proof}

\section{Non-simple points}\label{Sec:nonsimple}

\subsection{Non-simple points and gaps}
Consider a hyperbolic-like group $G$.
When $g\in G$ is a non-trivial element, we write $a(g)\in S^1$ for its attracting fixed point, and $r(g)$ for the repelling one. Clearly, we have the relation $r(g)=a(g^{-1})$.
For each point $p$ in $S^1$, we define the subsemigroup $H_+(p)$ of $G$ as 
\[H_+(p)=\{g\in G\setminus\{\id\}: r(g)=p\}.\]
Similarly, we set $H_-(p)$ as $H_-(p)=\{h^{-1}:h\in H_+(p)\}$. It follows that $p$ is the attracting fixed point of each element of $H_-(p)$. 
After Theorem \ref{Thm:elementary}, we have that for each $p\in S^1$, the stabilizer $\Stab{G}{p}$ is the disjoint union of $H_+(p)$, $H_-(p)$, and $\{\id\}$. 
We say that a point $p\in S^1$ is \emph{simple} if $\Stab{G}{p}$ is either trivial or cyclic, and \emph{non-simple}, otherwise.
When $H_+(p)$ is non-trivial, using \refthm{elementary} again, we can find a unique point $\overline{p}$ such that $H_+(\overline{p})=H_-(p)$. We call $\overline{p}$ the \emph{companion point} of $p$.
Note that $H_+(p)=H_-(\overline{p})$ and $H_-(p)=H_+(\overline{p})$.
Keeping the same notation as in the statement of \refthm{elementary}, the intervals $R(p)=\opi{p}{\overline{p}}$ and $L(p)=\opi{\overline{p}}{p}$ are called, respectively, the \emph{right side} and the \emph{left side} of $p$; when $I=R(p)$ is the right side of $p$ (respectively, when
$I=L(p)$ is the left side of $p$), the monotone map $m_{p,I}\colon I\to \RR$ is called the \emph{right} (respectively, \emph{left}) \emph{monotone map} at $p$, and the injective homomorphism $\theta_{p,I}\colon \Stab{G}{p}\to (\RR,+)$ is called the \emph{right} (respectively, \emph{left}) \emph{representation} of $\Stab{G}{p}$ at $p$; a gap of the monotone map $m_{p,I}$ is called a \emph{right} (respectively, \emph{left}) \emph{gap} of $p$. 
Finally, a \emph{gap} of $p$ is either a right or a left gap of $p$. Note that, by the choice of the linear orders in the statement of \refthm{elementary}, the semigroup $H_+(p)$ acts on each side of $p$ as a positive translation.
\begin{conv}
    When $p$ is simple, we always set $m_{p,I}$ to be a homeomorphism.
\end{conv}

\begin{rmk}\label{Rmk:wellDefineGap}
Considering the convention above, a gap of $p$ exists only if $p$ is non-simple. 
In this case, even though a monotone map and a representation of each side of $p$ are not uniquely determined, the core of any monotone map is the unique minimal invariant closed set for the action of $\Stab{G}{p}$ on the corresponding side; consequently, it does not depend on the monotone map and the representation. Similarly, the gaps do not depend on the monotone map and the representation. 
\end{rmk}

\begin{lem}\label{Lem:discreteSideGap}
Let $G$ be a hyperbolic-like group and $p$ a non-simple point. Then, the following are equivalent.
\begin{enumerate}[label=(\arabic*)]
    \item\label{i1:discreteSideGap} There is no right gap of $p$.
    \item\label{i2:discreteSideGap} Any right monotone map $m_{p,R(p)}\colon R(p)\to \RR$ is a homeomorphism.
    \item\label{i3:discreteSideGap} $\Stab{G}{p}$ is non-discrete on the right side of $p$.
    \item\label{i4:discreteSideGap} $G$ is non-discrete on the right side of $p$.
\end{enumerate}
The analogous statement holds when considering the left side of $p$.
\end{lem}

\begin{proof}
Since the left side of $p$ is the right side of $\overline{p}$, it is enough to discuss the case of the right side. After \refrmk{wellDefineGap}, a right gap of $p$ is a gap of any right monotone map at $p$, so \ref{i1:discreteSideGap} and \ref{i2:discreteSideGap} are equivalent. Next, consider a right monotone map and a right representation $\theta=\theta_{p,R(p)}$; note that the image of $\theta$ is dense in $\RR$, since we are assuming that $p$ is non-simple. Take any sequence $\{h_n\}_{n\in\NN}$ of non-trivial elements in $\Stab{G}{p}$ such that $\theta(h_n)$ converges to $0$, 
we have that the sequence of restrictions $\{h_n{\restriction_{R(p)}}\}_{n\in\NN}$ converges to $\id$ on $R(p)$ if and only if $m_{p,R(p)}$ is a homeomorphism. Thus, \ref{i2:discreteSideGap} and \ref{i3:discreteSideGap} are equivalent.
Finally, let us prove that \ref{i3:discreteSideGap} and \ref{i4:discreteSideGap} are equivalent. The implication \ref{i3:discreteSideGap}$\Rightarrow$\ref{i4:discreteSideGap} is clear, so let us assume \ref{i4:discreteSideGap}. If $G$ is elementary, \refthm{elementary} gives that $G=\Stab{G}{p}$, and there is nothing to prove. When $G$ is non-elementary, \refprop{locDiscrete} gives that $G$ is locally discrete, so the right side $R(p)$ must be a wandering interval for $G$, and actually a gap of the limit set $\Lambda(G)$ (by \reflem{limitSet}). Therefore, if $\{h_n\}_{n\in \NN}$ is any sequence of non-trivial elements of $G$ such that $\{h_n{\restriction_{R(p)}}\}_{n\in\NN}$ converges to $\id$ on $R(p)$, then $h_n$ must eventually fix $R(p)$, and consequently $h_n\in \Stab{G}{p}$ for any sufficiently large $n\in \NN$. This gives \ref{i4:discreteSideGap}$\Rightarrow$\ref{i3:discreteSideGap}.
\end{proof}

When a hyperbolic-like group is non-elementary, using \refprop{locDiscrete} as in the proof of \ref{i4:discreteSideGap}$\Rightarrow$\ref{i3:discreteSideGap} above, we see that if for a non-simple point there is no right (respectively, left) gap, then its right (respectively, left) side must be a gap of the limit set. Therefore, the converse of \refrmk{wellDefineGap} holds:

\begin{prop}\label{Prop:gapFromLimitSet}
    Let $G$ be a non-elementary hyperbolic-like group. Then, for any non-simple point $p$, there exists a gap of $p$. More precisely, if the right (respectively, left) side of $p$ is non-wandering, or equivalently, if it contains a fixed point of some non-trivial element of $G$, then there is a right (respectively, left) gap, and $G$ is discrete on $R(p)$ (respectively, on $L(p)$).
\end{prop}

\begin{rmk}
    The converse of the second statement in \refprop{gapFromLimitSet} is not true: the existence of a right gap does not imply that the right side intersects the limit set. 
\end{rmk}

\subsection{Cores at non-simple points}

\begin{defn}
    Let $G$ be a hyperbolic-like group. For any non-simple point $p$ in $S^1$, the \emph{core} at $p$, denoted by $\Core(p)$, is defined as
     \[\Core(p)=\Core(m_{p,R(p)}) \cup \Core(m_{p,L(p)})\cup \{p,\overline{p}\},\]
     where $m_{p,R(p)}$ (respectively, $m_{p,L(p)}$) is a right (respectively, left) monotone map.
\end{defn}
\begin{rmk}\label{Rmk:coreCantor}
By \refrmk{wellDefineGap}, the core at $p$ does not depend on the choice of the monotone maps: it is the closed set whose gaps are the gaps of $p$.
\end{rmk}

\begin{prop}\label{Prop:distinctCore}
Let $G$ be a hyperbolic-like group. Suppose that $p$ and $q$ are non-simple points with $p\notin\{q,\overline{q}\}$. Then, $\Core(p)\neq \Core(q)$.
\end{prop}
\begin{proof}
Assume that $\cC=\Core(p)=\Core(q)$, and consider the subgroup $H\le G$ generated by $\Stab{G}{p}$ and $\Stab{G}{q}$. Note that the condition $p\notin\{q,\overline{q}\}$ ensures that $H$ is non-elementary (by \refthm{elementary}).
By assumption, $H$ preserves $\cC$. After collapsing the closure of each gap of $\cC$ into a point, we obtain a new circle $\cS^1$, on which $H$ acts. 
The induced $H$-action on $\cS^1$ satisfies the following properties:
\begin{itemize}
    \item it is faithful (since any element in the kernel preserves each gap of $\cC$, and consequently has infinitely many fixed points),
    \item non-elementary (since the points $p$, $\overline{p}$, $q$ and $\overline{q}$ are in $\cC$),
    \item Möbius-like, that is, each element of $H$ is conjugate to an element in $\PSL(\RR
    )$ (since this property is preserved by the collapsing),
    \item every element has at least one fixed point (since fixed points of non-trivial elements cannot belong to gaps of $\cC$, as one can see by reproducing the argument for \reflem{limitSet}),
    \item the stabilizers $\Stab{G}{p}$ and $\Stab{G}{q}$ act non-discretely on both sides of the images of $p$ and $q$, respectively (since $p$ and $q$ are non-simple points). 
\end{itemize}
This contradicts \refprop{locDiscrete}. Thus, $\Core(p)\neq \Core(q)$.
\end{proof}

\subsection{Configurations of gaps}

 Here, we introduce one of the key ingredients for the proof of our main results. This relies on the analysis of fixed points of compositions from the appendix.

\begin{lem}\label{Lem:containGapGeneral}
    Let $G$ be a hyperbolic-like group. Suppose that $p$ and $q$ are non-simple points with $q\nin \{p,\overline{p}\}$, and that there is a gap $J$ of $q$ such that $\closure{J}\cap \{p,\overline p\}=\varnothing$. Then, at least one of the following conditions is satisfied:
    \begin{enumerate}
        \item $J$ is contained in some gap of $p$, or
        \item the pairs $\{p,\overline p\}$ and $\{q,\overline q\}$ are unlinked, and $J$ and $\{p,\overline p\}$ are not on the same side of $q$.
    \end{enumerate} 
\end{lem}

\begin{proof}
   Upon replacing $p$ and $q$ with their companion points $\overline{p}$ and $\overline{q}$, we can assume that $\overline J\subset R(p)\cap R(q)$.
    Assume by contradiction that there is no right gap of $p$ that contains $J$, or equivalently, that $J\cap \Core(p)\neq\varnothing$. Since $p$ is non-simple, we can find an element $h\in H_+(p)$ such that $J\cap h(J)\neq\varnothing$. Consider the intervals $K=J\setminus h^{-1}(\overline J)$ and $I=J\setminus h(\overline J)$. Then the images $h(K)$ and $h^{-1}(I)$ are in $R(p)$, and since they are adjacent to $J$, they both intersect $\Core(q)$. Since $q$ is non-simple and $J$ is a gap of $q$, we can find an element $g\in H_+(q)$ such that $g^{-1}(\overline J)\subset h^{-1}(I)$ and $g(\overline J)\subset h(K)$. With this condition, we have
    \[
    g^{-1}(\overline I)\subset g^{-1}(\overline J)\subset h^{-1}(I)\quad\text{and}\quad g(\overline K)\subset g(\overline J)\subset h(K),
    \]
    and thus $hg^{-1}(\overline I)\subset I$ and $gh^{-1}(\overline {h(K)})\subset h(K)$. Hence,
    \[a(hg^{-1})\in I\quad\text{and}\quad r(hg^{-1})=a(gh^{-1})\in h(K).\]
    Note that $h(K)\subset h(J)\subset R(q)$, and thus $\Fix(gh^{-1})\subset R(p)\cap R(q)$. Remark also that we have 
    \begin{equation}\label{Eqn:linkedGaps}
        p< a(hg^{-1})<r(hg^{-1})<\overline p<p.
    \end{equation}
     Now, we want to apply \refprop{configFix}: we look at the fixed point configurations for $g^{-1}$ (red), $h$ (blue), and the composition $hg^{-1}$ (purple) in Table \ref{Table:fixedPointsComposition} that are compatible with the conditions we have.
    When the pairs $\{p,\overline p\}$ and $\{q,\overline q\}$ are linked, we have to consider diagrams 1a and 1b; however, in both diagrams, the fixed points of $hg^{-1}$ are on different sides of $p$, and this is in contradiction with the condition $\Fix(hg^{-1})\subset R(p)$. For the unlinked case, recall that $r(h)=p$ and $r(g^{-1})=\overline q$, and that $\Fix(hg^{-1})\subset R(p)\cap R(q)=R(p)\cap L(\overline q)$; this reduces the possibilities to 
    2d/2e and 3b/3c.
    Note that the diagrams 2d/2e are not compatible with the condition \eqref{Eqn:linkedGaps}.
    So the only possible cases are 3b/3c, where we can see that $J$ and $\{p,\overline p\}$ are not on the same side of $q$.
\end{proof}

\section{Ping-pong dynamics in the linked case}\label{Sec:linked}

In this section we prove Theorem \ref{Thm1:linked}. This will be done through the analysis of fixed points of commutators of the form $[h,f]$, with $h\in \Stab{G}{p}$ and $f\in \Stab{G}{q}$. The preliminary combinatorial analysis is discussed in \refprop{fixCommutator} at the end of the paper and motivates the following definition.

\begin{defn}
    Let $G$ be a hyperbolic-like group, and consider points $p$ and $q$ with non-trivial stabilizers. Assume that $p<q<\overline p<\overline q<\overline p$. We say that a pair of elements $(h,f)\in H_+(p)\times H_+(q)$ is \emph{geometric} if we have the relations
   \[ \begin{array}{cccccccccccc}
        &p&<&h^{-1}(q)&<&r([f^{-1},h^{-1}])&<&a([f^{-1},h^{-1}])&<&f^{-1}(p)&\\
        <&q&<&f^{-1}(\overline p)&<&r([h,f^{-1}])&<&a([h,f^{-1}])&<&h(q)&\\
        <&\overline p&<&h(\overline q)&<&r([f,h])&<&a([f,h])&<&f(\overline p)&\\
        <&\overline q&<& f(p)&<& r([h^{-1},f])&<&a([h^{-1},f])&<&h^{-1}(q)&<&p.
    \end{array}\]
    See Figure \ref{Fig:fixCommutator} (left). We say that the points $p$ and $q$ are \emph{geometrically linked} if every pair of elements $(h,f)\in H_+(p)\times H_+(q)$ is geometric.
    When $p<\overline q<\overline p<q<p$, we say that $p$ and $q$ are \emph{geometrically linked} if $p$ and $\overline q$ are.
\end{defn}

The main result of this section is that when $G$ is a hyperbolic-like group, every pair of linked non-simple points is geometrically linked. From this, it is not difficult to find a nice proper ping-pong partition for the stabilizers; we work this out with \refcor{pingPongLinked} at the end of the section.

\begin{thm}\label{Thm:linkedCase}
    Let $G$ be a hyperbolic-like group. Assume that $p$ and $q$ are linked non-simple points, then $p$ and $q$ are geometrically linked. 
\end{thm}

Before proceeding to the proof of \refthm{linkedCase}, we describe how gaps of linked non-simple points are related by inclusion. Note that when $p$ and $q$ are linked non-simple points, \refprop{gapFromLimitSet} ensures that $p$ and $q$ have both right and left gaps. 

\begin{lem}\label{Lem:containGap}
    Let $G$ be a hyperbolic-like group. Assume that $p$ and $q$ are linked non-simple points, and let $J$ be a gap of $q$. Then, we have the following alternative:
    \begin{enumerate}[label=(\arabic*)]
        \item\label{Itm:containGap1} $J$ is also a gap of $p$, or
        \item\label{Itm:containGap2} there exists a gap $K$ of $p$ such that $J\subset K$ and $\closure{K}\cap \{q,\overline q\}\neq \varnothing$, or
        \item\label{Itm:containGap3} $\closure{J}\cap \{p,\overline p\}\neq \varnothing$, in which case $J$ contains gaps of $p$, and there exists a gap $K$ of $p$ such that $\closure{K}\cap \{q,\overline q\}\neq\varnothing$ and $\closure{K}\cap \partial J\neq \varnothing$.
    \end{enumerate}
\end{lem}
    
\begin{proof}
     First, let us assume that $\closure{J}\cap \{p,\overline p\}=\varnothing$, and prove that one of the first two conditions is satisfied. When this happens, by \reflem{containGapGeneral}, the gap $J$ is contained in a gap $K$ of $p$. If $\closure{K}\cap \{q,\overline q\}=\varnothing$, then we can apply \reflem{containGapGeneral} exchanging the roles of $p$ and $q$, and find a gap $J'$ of $q$ containing $K$. Therefore, we have $J\subset K\subset J'$, and since any two gaps of $q$ either coincide or are disjoint, we deduce that $J=K=J'$.

    Next, let us suppose $\closure{J}\cap \{p,\overline p\}\neq \varnothing$; up to replace $p$ with $\overline p$, we can assume that $J$ intersects any non-trivial neighborhood of $\overline p$. Since $p$ has gaps on both sides and gaps accumulate on $\overline p$, we deduce that $J$ contains gaps of $p$. To conclude the proof, assume by contradiction that there is no gap $K$ of $p$ such that $\closure{K}\cap \{q,\overline q\}\neq \varnothing$ and $\closure{K}\cap \partial J\neq \varnothing$. If so, without loss of generality, suppose $J\subset R(q)$ (if not, replace $q$ with $\overline{q}$), so that both components of $R(q)\setminus \overline J$ contain points of the core at $p$. Then we can find an element $h\in H_+(p)$ such that $h(q)$ and $h(\overline q)=\overline{h(q)}$ are contained in $R(q)\setminus \overline J$. Apply \reflem{containGapGeneral} to the non-simple points $h(q)$ and $q$ and the gap $J$: since $J$ and $\{h(q),\overline{h(q)}\}$ are contained in $R(q)$, the second possibility cannot occur, and thus we find a gap $J'$ of $h(q)$ that contains $J$. However, ${h(J)}$ is strictly contained in $J$, and $h(J)$ is also a gap of $h(q)$; therefore, we must have $h(J)=J'$. This provides the desired contradiction.
\end{proof}

\begin{proof}[Proof of \refthm{linkedCase}]
Without loss of generality, we can assume that $p$ and $q$ are non-simple points such that $p<q<\overline{p}<\overline{q}$. If $p$ and $q$ are not geometrically linked, then there is a pair $(h,f)\in H_+(p)\times H_+(q)$ which is not geometric, and thus $h$ and $f$ produce one of the four non-geometric configurations in the statement of \refprop{fixCommutator}. We detail the case \ref{Itm:nonGeometricII}, the others being similar: we assume by way of contradiction that $h$ and $f$ satisfy
\begin{equation}\label{Eqn:coincideQuadrant1}
q<h(q)<a([h,f^{-1}])<r([h,f^{-1}])<f^{-1}(\overline p)<\overline{p}.  
\end{equation}
Since $q<h(q)<f^{-1}(\overline{p})<\overline p$, we can
consider the largest open interval $I=\opi{s}{t}\subset \opi{q}{\overline p}$ such that $\Core(p)\cap  I=\Core(q)\cap I\neq \varnothing$.
After \reflem{containGap}, we have that either $q=s$, or $s\in \Core(p)\cap \Core(q)$ and there is a gap of $p$ whose closure contains $q$ and $s$; similarly, either $\overline p=t$, or $t\in \Core(p)\cap \Core(q)$ and there is a gap of $q$ whose closure contains $\overline p$ and $t$. Note that such an interval exists, since from \eqref{Eqn:coincideQuadrant1} we must have 
$q\le s<h(q)<f^{-1}(\overline p)<t\le \overline p$: indeed, otherwise, $q$ and $h(q)$ would be in the closure of the same gap of $p$, or $f^{-1}(\overline p)$ and $\overline p$ would be in the closure of the same gap of $q$, and both cases are not possible. In particular, we have $\opi{h(q)}{f^{-1}(\overline p)}\subset I$.

\begin{claim}\label{Clm:ifThereIsAGap}
    Assume that there is a gap $J$ of $p$ whose closure contains $h(q)$ and $f^{-1}(\overline p)$. Then $J$ is also a gap of $q$, and we have $q\in h^{-1}(J)$ and $\overline p\in f(J)$. Moreover, there is a left gap $L_q$ of $q$ (respectively, a left gap $L_p$ of $p$) whose closure intersects $p$ and $h^{-1}(J)$ (respectively, $\overline q$ and $f(J)$).
\end{claim}

\begin{proof}[Proof of the claim]
    Since $\opi{h(q)}{f^{-1}(\overline p)}\subset I$, the gap $J$ is contained in $I$ and it is also a gap of $q$. Note that $h^{-1}(J)$ is a gap of $p$ whose closure contains $q$, namely, $\opi{q}{h^{-1}f^{-1}(\overline{p})}\subset h^{-1}(J)$, and $f(J)$ is a gap of $q$ whose closure contains $\overline p$, namely, $\opi{fh(q)}{\overline{p}}\subset f(J)$. Therefore, we have $q\neq s$ and $t\neq \overline p$, and from \eqref{Eqn:coincideQuadrant1} we have
    \[
    q<s<f^{-1}(t)\le h(q)<f^{-1}(\overline p)\le h(s)<t<\overline p
    \]
    and $J=\opi{f^{-1}(t)}{h(s)}$. Also, we actually have $f^{-1}(t)\neq h(q)$ and $f^{-1}(\overline p)\neq h(s)$ because after \ref{Itm:containGap3} in \reflem{containGap}, there is a left gap $L_q$ of $q$ whose closure intersects $p$ and $h^{-1}(J)$, and similarly there is a left gap $L_p$ of $p$ whose closure intersects $\overline q$ and $f(J)$.  
\end{proof}

In fact, the key of the proof is to show the following.
\begin{goal*}
    For any non-geometric pair $(h,f)$ such that $[h,f]$ has fixed points in $\opi{q}{\overline p}$, there is a gap of $p$ whose closure contains $h(q)$ and $f^{-1}(\overline{p})$.
\end{goal*}
Let us see how this leads to the desired contradiction. 
After \refclm{ifThereIsAGap}, there are gaps $R_p$ and $L_p$ of $p$ (respectively, gaps $R_q$ and $L_q$ of $q$), with $q\in R_p$ and $\overline p\in R_q$ and $f^{-1}(R_q)=h(R_p)$, $\overline{R_p}\cap \overline{L_q}\neq\varnothing$, $\overline{R_q}\cap \overline{L_p}\neq\varnothing$. Writing $R_p=\opi{u_2}{v_2}$, take any $k\in H_+(p)$ such that $k(q)\in \opi{v_2}{h(u_2)}$. With this choice, we have that the pair $(k,f)$ is non-geometric, and there is no gap of $p$ whose closure contains $k(q)$ and $f^{-1}(\overline p)$. Therefore, after the Main Claim, the commutator $[k,f]$ cannot have fixed points in $\opi{p}{\overline q}$. Moreover, since $\opi{p}{q}\subset R_p\cup \overline{L_q}$, we must have $f^{-1}(p)\in R_p$ and $k^{-1}(q)\in L_q$, and consequently $p<k^{-1}(q)<f^{-1}(p)<q$. After \refprop{fixCommutator}, $[k,f]$ cannot have fixed points in $\opi{p}{q}$. A similar argument shows that $[k,f]$ cannot have fixed points in $\opi{\overline p}{\overline q}$. After \refprop{fixCommutator}, the fixed points of $[k,f]$ must be described by \ref{Itm:nonGeometricIV} in \refprop{fixCommutator}: we have $\Fix([k,f])\subset \opi{\overline q}{p}$. However, applying the Main Claim to the elements $k$ and $f$, we see that there is a gap of $p$ whose closure contains $k^{-1}(\overline q)$ and $f(p)$. However, this can hold for at most one such $k\in H_+(p)$, but there are infinitely many of them, because $p$ is non-simple. This gives the desired contradiction.

From now on until the end of the proof, we assume that there is no gap of $p$ whose closure contains $h(q)$ and $f^{-1}(\overline{p})$, and argue for a contradiction.

\begin{claim}\label{Clm:coincideQuadrant1}
    There is no gap of $p$ whose closure contains the points $h(s)$ and $f^{-1}(t)$. Moreover, we have
    \[q\le s<h(q)\le h(s)<f^{-1}(t)\le f^{-1}(\overline p)<t\le \overline p.\]
    When $t\neq \overline p$, there exists a gap of $p$ whose closure contains $f^{-1}(t)$ and $f^{-1}(\overline p)$.
\end{claim}

\begin{proof}[Proof of the claim]
    When $q=s$ and $\overline p=t$, this is a direct consequence of \eqref{Eqn:coincideQuadrant1} and the assumption that there is no gap of $p$ whose closure contains $h(q)$ and $f^{-1}(\overline p)$. Assume $q\neq s$, and let $J$ be the gap of $p$ whose closure contains $q$ and $s$. Then $h(J)$ is a gap of $p$, whose closure is contained in $\opi{s}{\overline p}$. After \eqref{Eqn:coincideQuadrant1} and the assumption, we must have $f^{-1}(\overline p)\in \opi{h(s)}{\overline p}$. Let us assume $t\neq\overline p$, otherwise there is nothing more to prove. Let $K$ be the gap of $q$ whose closure contains $t$ and $\overline p$; 
    then $f^{-1}(K)$ is a gap of $q$, whose closure is contained in $\opi{q}{t}$. Note that we cannot have $f^{-1}(K)\subset J$, because otherwise we get a contradiction with condition \eqref{Eqn:coincideQuadrant1}. Hence $f^{-1}(K)\subset \opi{s}{t}=I$, and therefore $f^{-1}(K)$ is also a gap of $p$. After the assumption, the gaps $h(J)$ and $f^{-1}(K)$, containing respectively $h(q)$ and $f^{-1}(\overline p)$, are disjoint, and thus condition \eqref{Eqn:coincideQuadrant1} yields the desired inequalities.
\end{proof}

Take a right monotone map $m\colon R(p)\to \RR$ such that $m(q)=m(s)=0$, and choose a right representation $\theta$ at $p$. For convenience, we write $r=f^{-1}(t)$ and $m(r)=m(f^{-1}(\overline p))=:\vartheta$. 
Consider the map $\hat f\colon [0,\vartheta) \to \RR$ defined as
$\hat f(m(x))=m(f(x))$ for all $x\in \ropi{s}{r}$. From the condition $I\cap \Core(p)=I\cap \Core(q)$, we see that $\hat f$ is well defined.
The next key argument is to show that $\hat{f}$ increases strictly faster than any positive translation on the image of $m$. For instance, see \reffig{fHatGraph} for a sketch of $\hat{f}$ when $\vartheta$ is finite and $t=\bar{p}$.

\begin{figure}[ht]
    \centering
    \includegraphics[width=0.5\linewidth]{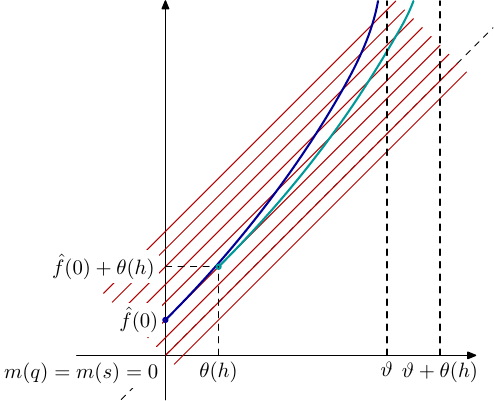}
    \caption{A sketch of  $\hat{f}$ for finite $\vartheta$ and $t=\overline{p}$: the red lines are the graphs of elements of $\Stab{G}{p}$. The dark blue line is the graph of $\hat{f}$ and the light blue curve is its conjugation by some element $h$ in $H_+(p)$.} 
    \label{Fig:fHatGraph}
    \end{figure}

\begin{claim}\label{Clm:tStrictMonotone}
    The function $\varphi\colon[0,\vartheta)\to \RR$ defined by $\varphi(z)=\hat{f}(z)-z$ is strictly increasing.
\end{claim}
\begin{proof}[Proof of the claim]
    To begin with, observe that $\varphi$ is continuous and that $\varphi(z)>0$ for every $z\in (0,\vartheta)$.
    We first prove that $\varphi$ is nowhere locally constant. For this, assume by contradiction that $\varphi$ is constant on some interval $C\subset (0,\vartheta)$ of non-empty interior. This means that $\hat{f}$ coincides with a translation on $C$. Then, for any $k\in \Stab{G}{p}$ such that $\theta(k)$ is sufficiently close to $0$, we have that $\hat{f}$ and the translation by $\theta(k)$ commute on some subinterval of $C$ of non-empty interior. This implies that the commutator $[f,k]$ has infinitely many fixed points, providing the desired contradiction.
    
    So, if $\varphi$ is not injective, we can find points $z_1<w<z_2$ in $(0,\vartheta)$ such that $\varphi(z_1)=\varphi(z_2)$ and $\varphi(w)\neq \varphi(z_1)$. If $\varphi(w)<\varphi(z_1)$ (respectively, if $\varphi(w)>\varphi(z_1)$), since $p$ is non-simple, we can take an element $k$ in $H_+(p)$ such that $\varphi(w)<\theta(k)<\varphi(z_1)$ (respectively, such that $\varphi(w)>\theta(k)>\varphi(z_1)$). On the other hand, by continuity of $\varphi$, there are points $v_1$ and $v_2$ such that $z_1<v_1<w<v_2<z_2$ and $\varphi(v_1)=\varphi(v_2)=\theta(k)$.
    For each $i\in \{1,2\}$, consider the preimage $V_i=m^{-1}(v_i)$, which is a closed interval; we have
    \[m(f(V_i))=\hat{f}(v_i)=v_i+\varphi(v_i)=v_i+\theta(k)=m(k(V_i)).\]
    Since the gaps of $p$ and $q$ in $I$ are the same, this implies that $f(V_i)=k(V_i)$, and thus $k^{-1}f$ fixes the endpoints of $V_i$. Since $k^{-1}f$ has only two fixed points, we deduce that $\Fix(k^{-1}f)\subset \opi{s}{r}$. However, \refprop{configFix} gives $\Fix(k^{-1}f)\cap \opi{\overline q}{p}\neq\varnothing$, and this provides the desired contradiction.

    Now, let us prove that $\varphi$ is strictly increasing. Since it is injective, it suffices to find two points $u,v\in \opi{s}{r}$ such that $m(u) < m(v)$ and $\varphi(m(u))\le \varphi(m(v))$. For this, as above, we note that for any $k\in H_+(p)$, the composition $k^{-1}f$ satisfies $\Fix(k^{-1}f)\cap \opi{\overline q}{p}=\{r(k^{-1}f)\}$ (\refprop{configFix}). As $p$ is non-simple, we can find $k\in H_+(p)$ such that $r(k^{-1}f)=:w$ is in the interval $\opi{s}{r}$ (otherwise, we would have that $\opi{s}{r}$ is contained in a gap of $p$, but this is not the case). As above, we remark that we necessarily have that $m^{-1}(w)$ is a singleton, because otherwise $k^{-1}f$ would preserve a non-trivial interval in $\opi{q}{\overline p}$.
    Now, for any $u\in \opi{s}{w}$, we have $q<f(u)<k(u)<\overline p$ and consequently $m(f(u))\le m(u)+\theta(k)$, or equivalently, $\varphi(m(u))\le \theta(k)$. Similarly, for any $v\in \opi{w}{r}$, we have $\varphi(m(v))\ge \theta(k)$. This gives $\varphi(m(u))\le \varphi(m(v))$. Finally, since $m^{-1}(w)$ is a singleton, we have $m(u)< m(w)< m(v)$. This shows that $t$ is strictly increasing.
\end{proof}

Now, write $w=a([h,f^{-1}])$ for the attractor of $[h,f^{-1}]$, and remember that from \eqref{Eqn:coincideQuadrant1} we have $w\in \opi{h(q)}{f^{-1}(\overline p)}$. Note that $w$ satisfies $f(w)=hfh^{-1}(w)$.

Let us first check that $w\notin \lopi{h(q)}{h(s)}$. When $s=q$ there is nothing to check, and thus we assume $s\neq q$ and, as in the proof of \refclm{coincideQuadrant1}, we let $J$ be the gap of $q$ whose closure contains $q$ and $s$. We will prove that $\Fix([h,f^{-1}])\cap \overline{h(J)}=\varnothing$. On the one hand, after \refclm{coincideQuadrant1}, the image $fh(J)$ is a gap of $p$ and $q$ contained in $I$, and we have $m(h(s))=\theta(h)\in (0,\vartheta)$, and thus $m(fh(s))=\hat f(\theta(h))$. On the other hand, the image $hf(s)$ is in $R(p)$, and we have $m(hf(s))=\theta(h)+m(f(s))=\theta(h)+\hat f(0)$. After \refclm{tStrictMonotone}, we have 
\[
m(fh(s))=\hat f(\theta(h))>\theta(h)+\hat f(0)=m(hf(s)),
\]
and consequently ${fh(J)}$ and ${hf(J)}$ are disjoint. This easily gives that the intervals $[h,f^{-1}]({h(J)})$ and ${h(J)}$ are disjoint, and thus $\Fix([h,f^{-1}])\cap \overline{h(J)}=\varnothing$.

Next, let us argue by contradiction that $w\notin \lopi{h(s)}{r}$. Note that the case $w=r$ is possible only when $t\neq \overline p$, and that in this case the value $\vartheta=m(r)$ is finite, so that the definitions of the maps $\hat f$ and $\varphi$ can be extended to $[0,\vartheta]$.
Now, if $w\in \lopi{h(s)}{r}$, then \eqref{Eqn:coincideQuadrant1} gives
$m(w)\in (\theta(h),\vartheta]$.
For simplicity, we write $\mu=m(w)$. The relation $f(w)=hfh^{-1}(w)$ yields
\[
\hat{f}(\mu)=\hat{f}(\mu-\theta(h))+\theta(h)=\mu+\varphi(\mu-\theta(h)),
\]
and thus $\varphi(\mu)=\varphi(\mu-\theta(h))$. Consequently, the function $\varphi$ is not injective, and after \refclm{tStrictMonotone} this provides the desired contradiction.

Finally, let us check that if $t\neq \overline p$, then $w\notin \opi{r}{f^{-1}(\overline p)}$. We argue by way of contradiction. As in the proof of \refclm{coincideQuadrant1}, let $K$ be the gap of $q$ whose closure contains $t$ and $\overline p$, so that $K':=f^{-1}(K)$ is a gap of $p$ and $q$ whose leftmost point is $r$. We want to prove that $[f^{-1},h](K')\cap K'=\varnothing$, or equivalently, that $hfh^{-1}(K')\cap K=\varnothing$. For this, note first that the point $h^{-1}(r)$ is in $I=\opi{s}{r}$ because $m(h^{-1}(r))=\vartheta-\theta(h)>0$. Hence $h^{-1}(K')$ is contained in $\opi{s}{r}$ and so it is a gap of $p$ and $q$. Applying $f$, the image $fh^{-1}(K')$ is contained in $\opi{s}{t}$ and so it is a gap of $p$ and $q$. Finally, $hfh^{-1}(K')$ is a gap of $p$. We conclude that if $hfh^{-1}(K')\cap K\neq \varnothing$, then $hfh^{-1}(\overline{K'})\subset K$, or equivalently, $[f^{-1},h](\overline{K'})\subset K'$. However, this contradicts the condition
\[w=a([h,f^{-1}])=r([f^{-1},h])\in \opi{r}{a([f^{-1},h])}\subset K',\]
that one easily gets from \eqref{Eqn:coincideQuadrant1} and $w\in \opi{r}{f^{-1}(\overline p)}$. This concludes the proof.
\end{proof}

\setcounter{claim}{0}

As announced before, we explain now how to define a proper ping-pong partition for the stabilizers of (geometrically) linked non-simple points.

\begin{cor}\label{Cor:pingPongLinked}
    Let $G$ be a hyperbolic-like group. Assume that $p$ and $q$ are linked non-simple points. Then, for any $\alpha\in \{p,\overline p\}$ (respectively, for any $\alpha\in \{q,\overline q\}$) there is a gap $I_\alpha$ of $q$ (respectively, of $p$) such that $\alpha\in I_\alpha$. Moreover, we have $S^1=\overline{I_p\cup I_{\overline p}\cup I_q\cup I_{\overline q}}$. In particular, setting
    \[U_p=(I_p\cup I_{\overline p})\setminus \overline{I_q\cup I_{\overline q}}\quad\text{and}\quad U_q=(I_q\cup I_{\overline q})\setminus \overline{I_p\cup I_{\overline p}},\]
    we have that $(U_p,U_q)$ is a proper ping-pong partition for $\Stab{G}{p}$ and $\Stab{G}{q}$.
\end{cor}

\begin{proof}
    Up to replace $p$ with $\overline p$, we can assume that the linked points are ordered as $p<q<\overline p<\overline q$. 
    Let us show that $\Core(q)$ does not accumulate on $\overline p$. After \refthm{linkedCase}, $p$ and $q$ are geometrically linked, and thus we have
    \[
        \overline p<\sup_{h\in H_+(p)}h(\overline q)\le \inf_{f\in H_+(q)}f(\overline p)<\overline q,
    \]
    where the supremum and the infimum are taken with respect to the linear order on $\opi{\overline p}{\overline q}$ induced by the circular order. Similarly, we have
    \[
        q<\sup_{f\in H_+(q)}f^{-1}(\overline p)\le \inf_{h\in H_+(p)}h(q)<\overline p.
    \]
    This shows that $\overline p$ is isolated in the $\Stab{G}{q}$-orbit of $\overline p$, and consequently $\overline p\notin \Core(q)$. This gives a gap $I_p$ of $q$ such that $p\in I_p$. Adapting this argument, one finds the other gaps $I_{\alpha}$ such that $\alpha\in I_\alpha$, for $\alpha\in \{\overline p,q,\overline q\}$.
    Let us write $I_\alpha=\opi{u_\alpha}{v_\alpha}$. We want to show that we have
    \begin{equation}\label{Eqn:pingPongLinked}
        p<u_q\le v_p<q<u_{\overline p}\le v_q\le \overline p<u_{\overline q}\le v_{\overline p}<\overline q<u_p\le v_{\overline q}<p.
    \end{equation}
    Indeed, this gives $S^1=\overline{I_p\cup I_{\overline p}\cup I_q\cup I_{\overline q}}$.

    Arguing for a contradiction, we assume that \eqref{Eqn:pingPongLinked} does not hold. We will discuss the case where $q<v_q<u_{\overline p}<\overline p$; the other cases can be treated similarly.
    By \reflem{containGap}, $J=\opi{v_q}{u_{\overline p}}$ is the largest open interval in $\opi{q}{\overline{p}}$ such that $\Core(p)\cap J= \Core(q)\cap J\neq\varnothing$. Since $p$ and $q$ are non-simple, we can take $h\in H_+(p)$ such that $h(I_q)\subset J$, and then take $f\in H_+(q)$ such that $fh(I_q)\subset J$. Note that $h(I_q)$ and $fh(I_q)$ are both gaps of $p$ and $q$. By \refprop{fixCommutator}, we have $\Fix([f^{-1},h^{-1}])\subset \opi{h^{-1}(q)}{f^{-1}(p)}\subset \opi{p}{q}$.
    We want to show that
    \[[f^{-1},h^{-1}](\opi{h^{-1}(q)}{f^{-1}(p)})\cap \opi{h^{-1}(q)}{f^{-1}(p)}=\varnothing,\]
    providing the desired contradiction. For this, since
    \[fh(\opi{h^{-1}(q)}{f^{-1}(p)})\subset fh(\opi{h^{-1}(q)}{q})=\opi{q}{fh(q)},\]
    it is enough to prove that
    $f^{-1}h^{-1}(\opi{q}{fh(q)})\subset \opi{f^{-1}(p)}{\overline p}$.
    To show this, note that on the one hand we have $h^{-1}(q)\in \opi{p}{q}$ and $f^{-1}(\opi{p}{q})= \opi{f^{-1}(p)}{q}$, hence $f^{-1}h^{-1}(q)\in \opi{f^{-1}(p)}{q}$. On the other hand, we have $fh(q)\in fh(I_q)$; the image $h^{-1}(fh(I_q))$ is contained in $J$ because $h^{-1}(h(I_q))=I_q$ and  $fh(I_q)$ is between $h(I_q)$ and $I_{\overline p}$. Finally, we have $f^{-1}(J)\subset \opi{q}{\overline p}$, and we conclude that $f^{-1}(h^{-1}fh(q))\in f^{-1}(J)\subset \opi{q}{\overline p}$. This easily gives the desired inclusion $f^{-1}h^{-1}(\opi{q}{fh(q)})\subset \opi{f^{-1}(p)}{\overline p}$.

    To conclude, note that for any $h\in H_+(p)$ and $f\in H_+(q)$ we have
    \begin{align*}
        h^{-1}(I_q\cup I_{\overline q})\subset I_p\setminus \overline{I_q\cup I_{\overline q}},\quad h(I_q\cup I_{\overline q})\subset I_{\overline p}\setminus \overline{I_q\cup I_{\overline q}},\\
        f^{-1}(I_p\cup I_{\overline p})\subset I_q\setminus \overline{I_p\cup I_{\overline p}},\quad f(I_p\cup I_{\overline p})\subset I_{\overline q}\setminus \overline{I_p\cup I_{\overline p}},
    \end{align*}
    and this shows that $(U_p,U_q)$ is a proper ping-pong partition for $\Stab{G}{p}$ and $\Stab{G}{q}$.
\end{proof}
\setcounter{claim}{0}

\section{Ping-pong dynamics in the unlinked case}\label{Sec:unlinked}

\subsection{Possible configurations of gaps in outside intervals}

In this section we prove Theorem~\ref{Thm1:unlinked}. That is, we deal with the case where $p$ and $q$ are unlinked non-simple points. Upon replacing $q$ with $\overline q$, it is enough to consider the case $p<\overline{q}<q<\overline{p}$.

\begin{rmk}
    Assume $p$ and $q$ are unlinked non-simple points with $p<\overline{q}<q<\overline{p}$.
    By \refprop{gapFromLimitSet}, $G$ is discrete on the right side of both $p$ and $q$, since $R(p)$ and $R(q)$ contain fixed points of elements of $\Stab{G}{q}$ and $\Stab{G}{p}$, respectively. Therefore, $p$ and $q$ have right gaps.
\end{rmk}

The next result gives a criterion for having gaps on the other sides, and describe their possible configurations.

\begin{prop}\label{Prop:outsideConfig}
     Let $G$ be a hyperbolic-like group. Assume that $p$ and $q$ are unlinked non-simple points with $p<\overline{q}<q<\overline{p}$. If $\opi{\overline{q}}{q}\cap \Core(p)\neq \varnothing$, then there exists a unique left gap $L$ of $q$ such that $\Core(p)\cap L\neq \varnothing$. Moreover, it satisfies $\Core(p)\cap (\cldi{\overline q}{q}\setminus \overline L)=\varnothing$. In other terms, writing $L=\opi{u}{v}$, there exist right gaps $R_1$ and $R_2$ of $p$ such that $\ropi{\overline q}{u}\subset R_1$ and $\lopi{v}{q}\subset R_2$.
\end{prop}
\begin{proof}
    First, let us see that $q$ has left gaps. Since we are assuming $\opi{\overline{q}}{q}\cap \Core(p)\neq \varnothing$, as $p$ is non-simple, we can take $h\in H_+(p)$ such that $\overline q<h(\overline q)<{q}<h({q})<\overline{p}$.
    Since $h(\overline q)$ is also non-simple, by \refprop{gapFromLimitSet}, $G$ is discrete on $L(q)=\opi{\overline q}{{q}}$ and there is a left gap of $q$.

    Second, we show that there exists a left gap $L$ of $q$ such that $\Core(p)\cap L\neq \varnothing$. Assume by contradiction that any left gap of $q$ is contained in a gap of $p$. 
    From \reflem{containGapGeneral}, we can consider the maximal open interval $J=\opi{\alpha}{\beta}$ in $L(q)$ such that $\Core(p)\cap J=\Core(q)\cap J\neq\varnothing$.
    Note that after \reflem{containGapGeneral}, if $\overline q\neq \alpha$ (respectively, if ${q}\neq \beta$), then there is a gap of $p$ that contains $\opi{\overline q}{\alpha}$ (respectively, $\opi{\beta}{{q}}$).
    By non-simplicity of $p$, we can take $f\in H_+(p)$ such that 
    \[\overline q\le \alpha<f(\overline q)\le f(\alpha)<\beta\le  q<f(\beta)\le f({q}).\]
    Note that by the properties of $J$ and since $f(\Core(q))=\Core(f(q))$, the interval $K=\opi{f(\alpha)}{\beta}$ satisfies $\Core(q)\cap K=\Core(f(q))\cap K\neq\varnothing$.
    However, applying \refcor{pingPongLinked} to the linked non-simple points $q$ and $f(q)$, we see that such an interval $K$ cannot exist. This gives the desired contradiction.

    Next, let us show that $L$ satisfies $\Core(p)\cap (\opi{\overline q}{ q}\setminus \overline L)=\varnothing$. Note that this property ensures that $L$ is unique. Write $L=\opi{u}{v}$, and assume by way of contradiction that $\Core(p)\cap \opi{\overline q}{u}\neq\varnothing$. Since we also have $\Core(p)\cap L\neq\varnothing$ and $p$ is non-simple, we can take an element $f\in H_+(p)$ such that 
    \begin{equation}\label{Eqn:endCollapseParallel}
        p<\overline q<f(\overline q)<u<f(u)<v<f(v)<f({q})<\overline{p}.
    \end{equation}
    With such a choice, the non-simple points $q$ and $f(q)$ are linked, and $\overline {L}\cap\{f(\overline q),f(q)\}=\varnothing$. After \reflem{containGapGeneral}, the gap $L$ of $q$ is contained in some gap of $f(q)$. This provides the desired contradiction, since $f(L)$ is a gap of $f(q)$ intersecting $L$, but not containing it (because of condition \eqref{Eqn:endCollapseParallel}). The same strategy shows that $\opi{v}{{q}}\cap \Core(p)=\varnothing$. 

    Finally, let us argue by contradiction that $\overline q\notin \Core(p)$. Assume this is not the case, then the gap $R_1$ of $p$ from the statement is of the form $\opi{\overline q}{v_1}$, with $\overline q<u\le v_1<v$.
    Take an element $f$ in $H_-(q)$. Then, $f(R_1)$ is a right gap of $f(p)$ such that $R_1\subsetneq f(R_1)\subsetneq L(q)$. On the other hand, we have that either $f(p)$ and $p$ are linked, or $\{f(p),f(\overline p)\}\subset \opi{q}{\overline p}\subset R(p)$. In both cases, by \reflem{containGapGeneral}, we must have $f(R_1)=R_1$. This gives the desired contradiction. A similar argument shows that ${q}\notin \Core(p)$.
\end{proof}

\subsection{Possible configurations of gaps in inside intervals}

Here we consider the right gaps of $p$ and $q$ when $p<\overline{q}<q<\overline{p}$. 
The next lemma is a direct consequence of \reflem{containGapGeneral}.

\begin{lem}\label{Lem:coincideGapAntiParallel}
    Let $G$ be a hyperbolic-like group. Assume that $p$ and $q$ are unlinked non-simple points with $p<\overline{q}<q<\overline{p}$. If $R$ is a gap of $p$ such that $\closure{R}\subset R(p)\cap R(q)$, then either $R$ is also a gap of $q$, or $R$ is contained in a gap of $q$ whose closure intersects $\{p,\overline{p}\}$.
\end{lem}
\begin{proof}
    Since $\closure{R}\cap\{q,\overline{q}\}=\emptyset$, and since $R$ and $\{q,\overline {q}\}$ are in $R(p)$, by \reflem{containGapGeneral}, there is a gap $R'$ of $q$ that contains $R$. Moreover, applying \reflem{containGapGeneral} to $R'$, we have that either $\closure{R'}$ intersects $\{p,\overline{p}\}$ or $R=R'$.
\end{proof}

\begin{lem}\label{Lem:endCollapseAntiParallel1}
    Let $G$ be a hyperbolic-like group. Assume that $p$ and $q$ are unlinked non-simple points with $p<\overline{q}<q<\overline{p}$. If there is a right gap $R_p=\opi{u}{v}$ of $p$ such that $u<\overline{q}\leq v\le q$, then there is a right gap $R_q$ of $q$ such that $\opi{p}{u}\subset R_q$.
\end{lem}
\begin{proof}
    We adapt the argument given at the end of the proof of \refprop{outsideConfig} to show that $\overline q\notin \Core(p)$.
    Assume by way of contradiction that $\opi{p}{u}\cap\Core(q)\neq\varnothing$. Then, we can take an element $f\in H_-(q)$ such that $f(p)<p<f(u)<u$. Then $f(R_p)$ is a right gap of $f(p)$ such that $\overline{f(R_p)}\cap \{p,\overline p\}=\varnothing$, and we have $f(R_p)\cap R_p\neq\varnothing$, but $f(R_p)\neq R_p$. Since $\overline{R_p}\cap \{f(p),f(\overline p)\}=\varnothing$, according to \reflem{containGapGeneral}, we must have $\{p,\overline p\}\subset L(f(p))$. However, after the choice of $f$, we have $p\in R(f(p))$, and this provides the desired contradiction.
\end{proof}

\begin{lem}\label{Lem:noHexagon}
    Let $G$ be a hyperbolic-like group. Assume that $p$ and $q$ are unlinked non-simple points with $p<\overline{q}<q<\overline{p}$. Then, at least one of the following occurs: $L(p)\cap\Core(q)=\varnothing$, or $L(q)\cap \Core(p)=\varnothing$.
\end{lem}
\begin{proof}
    Assume by way of contradiction that $L(p)\cap\Core(q)\neq\varnothing$ and $L(q)\cap \Core(p)\neq\varnothing$.
    By \refprop{outsideConfig}, there are six intervals $I_i=\opi{u_i}{v_i}$, with $i\in\ZZ/6\ZZ$, which satisfy the following:
    \begin{itemize}
        \item $I_1$ and $I_3$ are right gaps of $p$, and $I_5$ is a left gap of $p$;
        \item $I_4$ and $I_5$ are right gaps of $q$, and $I_2$ is a left gap of $q$;
        \item $\ropi{\overline{q}}{u_2}\subset I_1$, $\lopi{v_2}{q}\subset I_3$, $\ropi{\overline{p}}{u_5}\subset I_4$, $\lopi{v_5}{p}\subset I_6$.
    \end{itemize}
    After \reflem{endCollapseAntiParallel1}, these properties yield
    \[\opi{v_6}{\overline{q}}\subset I_1,\ \opi{q}{u_4}\subset I_3,\ \opi{v_3}{\overline{p}}\subset I_4,\ \opi{p}{u_1}\subset I_6.\]
    As a summary, the intervals $\overline{I_i}$ cover the circle, and we have $u_i\leq v_{i-1}<u_{i+1}\leq v_i$ for all $i\in \ZZ/6\ZZ$. 
    
    Now, take elements $h\in H_+(p)$ and $f\in H_+(q)$ such that $h(I_1)\subset I_2$ and $f(I_4)\subset I_5$.
    Observe that this also gives $h(I_5)\subset I_4$ and $f(I_2)\subset I_1$; hence, 
    \[fh(I_1)\subset f(I_2)\subset I_1\text{ and }fh(I_5)\subset f(I_4)\subset I_5.\]
    This implies that $\closure{I_1}$ and $\closure{I_5}$ contain the attracting fixed point of $fh$. However, $\closure{I_1}\cap \closure{I_5}=\emptyset$, and this provides the desired contradiction.
\end{proof}

When $p$ and $q$ are unlinked non-simple points and we can cover the circle with the closure of two gaps, it is not difficult to produce a proper ping-pong partition.

\begin{prop}\label{Prop:pingPongUnlinked}
    Let $G$ be a hyperbolic-like group. Suppose that $p$ and $q$ are unlinked non-simple points with $p<\overline q<q<\overline p$ and assume that there are right gaps $R_p$ and $R_q$ of $p$ and $q$, respectively, such that $S^1=\overline{R_p\cup R_q}$. Then, setting $U_p=R_q\setminus \overline{R_p}$ and $U_q=R_p\setminus \overline{R_q}$, we have that $(U_p,U_q)$ is a proper ping-pong partition for $\Stab{G}{p}$ and $\Stab{G}{q}$.
\end{prop}
\begin{proof}
    Take $f\in \Stab{G}{q}\setminus \{\id\}$. As $R_q$ is a gap of $q$, the image $f(R_q)$ is also a gap of $q$, whose closure is disjoint from $\overline{R_q}$. Thus, it must be contained in $R_p\setminus \overline{R_q}=U_q$. Arguing similarly, we have $h(R_p)\subset U_p$ for any $h\in \Stab{G}{p}\setminus \{\id\}$. Since $U_p\subset R_q$ and $U_q\subset R_p$, we have the desired result.
\end{proof}

We are now ready to discuss the main result of the section which, combined with \reflem{noHexagon}, provides a complete description of the gaps configurations when $p$ and $q$ are unlinked. 

\begin{thm}\label{Thm:unlinkedCase}
    Let $G$ be a hyperbolic-like group. Assume that $p$ and $q$ are unlinked non-simple points with $p<\overline{q}<q<\overline{p}$. If there is a right gap $R_p=\opi{u}{v}$ of $p$ containing $L(q)$, namely, if $u\leq \overline{q}<q\leq v$, then there are right gaps $R_1$ and $R_2$ of $q$ such that $\opi{p}{u}\subset R_1$ and $\opi{v}{\overline{p}}\subset R_2$.
    
    Moreover, when $R_1\neq R_2$, there is a left gap $L_p=\opi{\alpha}{\beta}$ of $p$ such that $\opi{\beta}{p}\subset R_1$ and $\opi{\overline{p}}{\alpha}\subset R_2$. Otherwise, we have $R_1=R_2$ and $S^1=\overline {R_1\cup R_p}$. 
\end{thm}
\begin{proof}
    We discuss separately the cases according to whether $L(p)\cap \Core(q)$ is empty or not. The case where $L(p)\cap \Core(q)\neq\varnothing$ is the easiest one. When this happens, after \refprop{outsideConfig}, there is a unique left gap $L_p=\opi{\alpha}{\beta}$ of $p$ and right gaps $R_1$ and $R_2$ of $q$ such that $\lopi{\beta}{p}\subset R_1$ and $\ropi{\overline{p}}{\alpha}\subset R_2$.
    After \reflem{endCollapseAntiParallel1}, we have $\opi{p}{u}\subset R_1$ and $\opi{v}{\overline{p}}\subset R_2$, as desired.
   From now on, we consider the case where $L(p)$ is contained in some right gap $R_q=\opi{x}{y}$ of $q$. Let us prove by contradiction that $\overline{R_q\cup R_p}=S^1$. Hence, let us assume that this is not true, and let us first discuss the case where $\overline{R_p}\cap \overline{R_q}\neq \varnothing$; upon exchanging the roles of $p$ and $q$, we can assume that $\overline{R_p}\cap \overline{R_q}\subset \opi{p}{\overline q}$. That is, we have
   \[p<u\leq y<\overline{q}<q\leq v<x\leq \overline{p}.\]
   Set $K:=S^1\setminus (\overline{R_p}\cup \overline{R_q})=\opi{v}{x}$.
   By \reflem{coincideGapAntiParallel}, we have $\Core(p)\cap K=\Core(q)\cap K$ (and by the choice of $K$ these intersections are non-empty). As $p$ is non-simple, we can take an element $h$ in $H_+(p)$ such that $h(R_p)=\opi{h(u)}{h(v)}\subset K$, and we have that $h(R_p)$ is a gap of both $p$ and $q$. Then, using now that $q$ is non-simple, we can take
   an element $f$ in $H_-(q)$ such that $f(R_q)=\opi{f(x)}{f(y)}\subset \opi{h(v)}{x}$, and again we have that $f(R_q)$ is a gap of both $p$ and $q$. We reach the desired contradiction with the following claim.
    \begin{claim*}
        The composition $fh^{-1}$ has no fixed point.
    \end{claim*}
    \begin{proof}[Proof of the claim]
    After the choices of $h$ and $f$, we have
    \[fh^{-1}(\cldi{v}{h(v)})\subset f(\lopi{p}{v})\subset \opi{f(x)}{v}\subset \opi{h(v)}{v},\]
    and consequently $fh^{-1}(\cldi{v}{h(v)})\cap \cldi{v}{h(v)}=\varnothing$.
    A similar argument gives
    \[hf^{-1}(\cldi{f(x)}{x})\cap \cldi{f(x)}{x}=\varnothing.\]
    Moreover, \[fh^{-1}(R_p)\cap R_p\subset f(R_q)\cap R_p=\varnothing \text{ and } hf^{-1}(R_q)\cap R_q\subset h(R_p)\cap R_q=\varnothing.\]
    Thus, we must have $\Fix(fh^{-1})\subset \opi{h(v)}{f(x)}$. However, consider any point $w\in \opi{h(v)}{f(x)}$. In particular, we have $q\le v< h(v)<w <\overline{p}$. Applying $h^{-1}\in H_-(p)$, we get
        \[q\leq v=h^{-1}(h(v))<h^{-1}(w)<w<\overline{p}.\]
    In particular, we have $q<h^{-1}(w)<w<\overline p$, so that when applying $f\in H_-(q)$, we get $q\leq fh^{-1}(w)<w$. Thus, $w$ is not a fixed point of $fh^{-1}$. This shows that $\opi{h(v)}{f(x)}$ contains no fixed point of $fh^{-1}$, and consequently $fh^{-1}$ has no fixed point, as desired. \qedhere
    \end{proof}

    It remains to discuss the case where $\overline{R_p}\cap \overline{R_q}= \varnothing$. If so, then $S^1\setminus (\overline{R_p}\cup \overline{R_q})=\opi{v}{x} \cup \opi{y}{u}$, with $v\neq x$ and $y\neq u$. Write $K:=\opi{v}{x} \cup \opi{y}{u}$, then by \reflem{coincideGapAntiParallel} we have $\Core(p)\cap K=\Core(q)\cap K$, and these intersections are non-empty.
    If $R_p=L(q)$ and $R_q=L(p)$, then we can proceed as in the proof of \refprop{distinctCore}: consider the induced action of the subgroup $H$ generated by $\Stab{G}{p}$ and $\Stab{G}{q}$ on the topological circle obtained by collapsing gaps of $\Core(p)\cap K$, and use \refprop{locDiscrete} to find an element in $H$ without fixed points. 
    Hence, we must have $R_p\neq L(q)$ or $R_q\neq L(p)$. Without loss of generality, we may say that $R_q\neq L(p)$. Note that after \reflem{endCollapseAntiParallel1}, we must have $R_q\supset \overline{L(q)}$ (otherwise $\overline {R_p}\cap \overline {R_q}\neq\varnothing$).
    Since $\Core(p)\cap \opi{y}{u}\neq\varnothing$ and $\Core(p)\cap \opi{v}{x}\neq \varnothing$, we can take an element $h\in H_+(p)$ such that
    $y<h(y)<u<x< h(x)<\overline{p}$
    and $h(R_p)\subset \opi{v}{x}$. With this choice, $\overline{h(R_q)}$ and $\{h(q),h(\overline q)\}$ are in $R(q)$, and $\overline{R_q}$ and $\{q,\overline q\}$ are in $R(h(q))$. Moreover, the gap $R_q$ of $q$ intersects the gap $h(R_q)$ of $h(q)$, but they do not coincide. This contradicts \reflem{containGapGeneral}.
\end{proof}

We obtain a stronger result when the non-simple points lie in the same orbit. 
\begin{prop}\label{Prop:conjFrameExist}
Let $G$ be a hyperbolic-like group. Suppose that $p$ and $q$ are unlinked non-simple points with $p<\overline{q}<q<\overline{p}$. If $g(p)\in \{q,\overline q\}$ for some $g\in G$, then there are right gaps $R_p$ and $R_q$, of $p$ and $q$ respectively, such that $S^1=\overline {R_p\cup R_q}$.
\end{prop}
\begin{proof}
    Upon exchanging the roles of $p$ and $q$, by \reflem{noHexagon}, we can assume that there is a right gap $R_p$ of $p$ containing $L(q)$. 
    We want to show that when considering \refthm{unlinkedCase} applied to the points $p$ and $q$, we must have $R_1=R_2$ (for the proof, the intervals $R_p$, $L_p$, and $R_i$ are as in the statement of \refthm{unlinkedCase}).
    Arguing by contradiction, we assume $R_1\neq R_2$. Note that when $g(p)\in \{q,\overline q\}$, then \refcor{commute} ensures that $g(\overline p)=\overline{g(p)}\in \{q,\overline q\}$. Let us first consider the case $g(p)=\overline q$. In this case, we have $g(L(p))=R(q)$, and therefore $r(g)\in L(p)$ and $a(g)\in L(q)$. From this, we see that the interval $g(L_p)$, which is a gap of $q$, is contained in $R(q)$. By non-simplicity of $q$, since $\Core(q)\cap L_p\neq\varnothing$, we can find an element $f\in \Stab{G}{q}$ such that $fg(L_p)\subset L_p$. We conclude that $a(fg)\in L_p$. However, according to \refprop{configFix} applied to the elements $f$ and $g$, whose fixed points are linked, the attracting fixed point of $fg$ must be as described by the diagrams 1a or 1b in Table \ref{Table:fixedPointsComposition}: since $a(g)\in L(q)$, we conclude that $a(fg)\in L(q)$ as well. This contradicts the previous observation that $a(fg)\in L_p$.
    
    Let us next consider the case $g(p)=q$, that is, we assume $g(L(p))=L(q)$. If so, on the one hand, from $g(L(p))\cap L(p)=L(q)\cap L(p)=\varnothing$, we deduce that $g(L(q)\cap L(p))=g(L(q))\cap L(q)=\varnothing$, and therefore, $g(R_p)$, which is a gap of $q$ that contains $g(L(p))$, is contained in $R(q)=S^1\setminus\overline{L(q)}$. By non-simplicity of $q$, since $\Core(q)\cap L_p\neq\varnothing$, we can find an element $f\in \Stab{G}{q}$ such that
    \begin{equation}\label{Eqn:conjFrameExist1}
        \overline{fg(R_p)}\subset L_p.
    \end{equation}
    On the other hand, we have that the interval $g(L_p)\subset g(L(p))=L(q)$ is a left gap of $q$, and thus so is $fg(L_p)$. We deduce that
    \begin{equation}\label{Eqn:conjFrameExist2}
        \overline{fg(L_p)}\subset L(q)\subset R_p.
    \end{equation}
    Combining \eqref{Eqn:conjFrameExist1} and \eqref{Eqn:conjFrameExist2} together, we deduce that $(fg)^2$ has an attracting fixed point in $L_p$ and another one in $R_p$. This contradicts the assumption that $G$ is hyberbolic-like.
    \qedhere   
\end{proof}

In fact, \refprop{conjFrameExist} can be used to produce a proper ping-pong partition for the remaining unlinked cases. 

\begin{cor}\label{Cor:parallelPingPong}
    Let $G$ be a hyperbolic-like group. Suppose that $p$ and $q$ are unlinked non-simple points with $p<\overline{q}<q<\overline{p}$. Suppose that there are intervals $I_i=\opi{u_i}{v_i}$, with $i\in \ZZ/4\ZZ$, satisfying the following:
    \begin{itemize}
        \item $I_1$ and $I_3$ are gaps of $p$, 
        \item $I_2$ and $I_4$ are gaps of $q$,
        \item we have
        \[
        p<u_1\le v_4<\overline q<u_2\le v_1<u_3\le v_2<q<u_4\le v_3<\overline p<p.
        \]
    \end{itemize}
    Then, setting $U_p=I_2\cup I_4$ and $U_q=(I_1\cup I_3)\setminus \overline {I_2\cup I_4}$, we have that $(U_p,U_q)$ is a proper ping-pong partition for $\Stab{G}{p}$ and $\Stab{G}{q}$.
\end{cor}
\begin{proof}
    For any $f\in H_+(q)$,  since $I_2$ and $I_4$ are gaps of $q$, we have
    \[f(U_p)=f(I_2\cup I_4)\subset I_1\setminus \overline{I_2\cup I_4}\subset U_q\quad \text{and}\quad f^{-1}(U_p)=f^{-1}(I_2\cup I_4)\subset I_3\setminus \overline {I_2\cup I_4}\subset U_q.\]
    Similarly, for any $h\in H_+(p)$, we have $h(I_3)\subset I_4\subset U_q$ and $h^{-1}(I_1)\subset I_4\subset U_q$. If $h(I_1)$ is not contained in $I_2\cup I_4=U_q$, then we must have $h(I_1)=I_3$ (and consequently $h^{-1}(I_3)=I_1$).
    Assume that this is the case.   
   Note that we have $\overline q\in h(I_4)$ and $h({q})\in h(I_3)\subset I_4$. Since $h(\overline q)\in h(I_1)=I_3$, one of the following occurs:
   \begin{enumerate}[label=(\arabic*)]
        \item $h(\overline q)\in \opi{{q}}{v_3}$;
        \item $h(\overline q)\in \opi{u_3}{{q}}$ (note that we cannot have $h(\overline q)=q$, otherwise $\Stab{G}{q}\neq \Stab{G}{\overline q}$, contradicting \refcor{commute}).
   \end{enumerate}
   In the first case, the non-simple points $q$ and $h(q)$ are unlinked, and the configuration of gaps is described by \refprop{conjFrameExist}: since $\overline q$ and $h(q)$ are contained, respectively, in the gap $h(I_4)$ of $h(q)$, and in the gap $I_4$ of $q$, we have $S^1=\overline{I_4\cup h(I_4)}$. This ensures that $h(I_1\setminus \overline{I_4})\subset I_4\subset U_q$. Similarly, from $S^1=\overline{h^{-1}(I_4)\cup I_4}$, we get that $h^{-1}(I_3\setminus \overline{I_4})\subset I_4\subset U_q$. Thus, in the first case we have shown that $(U_p,U_q)$ is a proper ping-pong partition for $\Stab{G}{p}$ and $\Stab{G}{q}$.
   
   In the second case, the non-simple points $q$ and $h(q)$ are linked, and consequently their gaps must be arranged according to \refcor{pingPongLinked}. Let us see which restrictions we get from this.
   
    \begin{claim}
        If $h(\overline q)\in \opi{u_3}{q}$, then $h(\overline q)\in \opi{u_3}{v_2}=I_2\cap I_3$.
    \end{claim}

    \begin{proof}[Proof of the claim]\label{Clm:parallelPingPong}
        We argue by way of contradiction. Note that according to \refcor{pingPongLinked} applied to $q$ and $h(q)$, we cannot have $v_2=h(q)$, because otherwise $h(q)$ would belong to the boundary of a gap of $q$. Thus, if $h(\overline q)\notin \opi{u_3}{v_2}$, we must have $h(\overline q)\in \opi{v_2}{q}$. Then, according to \refcor{pingPongLinked}, there is a gap $K_1=\opi{\alpha_1}{\beta_1}$ of $q$, distinct from $I_2$, such that $h(\overline q)\in K_1$. Since $q$ is non-simple, we can consider an element $f\in H_+(q)$ such that $f(K_1)\subset \opi{v_2}{\alpha_1}$. Using \refcor{pingPongLinked} again (for $q$ and $h(q)$), we take a gap $K_2=\opi{\alpha_2}{\beta_2}$ of $q$ such that $q\in h(K_2)$ and $S^1=\overline{h(I_4)\cup K_1\cup h(K_2)\cup I_4}$; note that  $K_2=I_2$ or $K_2\subset \opi{\overline q}{u_2}\subset I_1$, and in either case we have $I_2\cup K_2\subset h(I_4)$. Since $h(K_2)$ contains $\opi{q}{u_4}$ and $f\in H_+(q)$, we see that
        \begin{equation}\label{Eqn:parallelPingPong3}
            fh(K_2)\supset \overline{I_4}.
        \end{equation}
        On the other hand, let us see where $fh(I_4)$ is: the endpoint $h(u_4)$ of $h(I_4)$ belongs to $\overline K_1$, and the other endpoint $h(v_4)$ belongs to $\overline {I_4}$; then $fh(u_4)\in \overline{f(K_1)}$ and $fh(v_4)\in f(I_4)$. We deduce that
        \begin{equation}\label{Eqn:parallelPingPong4}
            fh(I_4)\supset \overline{K_2}.
        \end{equation}
        Combining \eqref{Eqn:parallelPingPong3} and \eqref{Eqn:parallelPingPong4} together, we see that $(fh)^2$ has repelling fixed points in both $K_2$ and $I_4$. Since these intervals are disjoint, this contradicts the assumption that $G$ is hyperbolic-like.
    \end{proof}
   
   \begin{claim}\label{Clm:pingPongLinked2}
        If $h(\overline q)\in \opi{u_3}{q}$, then $q\in h(I_2)$. Consequently, $h(I_1\setminus \overline{I_2})\subset I_2$ and $h^{-1}(I_3\setminus \overline {I_2})\setminus I_2$.
   \end{claim}
   \begin{proof}[Proof of the claim]
        Under the assumption, we know from \refclm{parallelPingPong} that $h(\overline q)\in I_2\cap I_3$.
       We adapt the proof of the previous claim. If, by contradiction, we have $q\notin h(I_2)$, then after \refcor{pingPongLinked} there is a gap $J=\opi{\alpha}{\beta}$ of $q$, distinct from $I_2$, such that $q\in h(J)$. From the condition $h(\overline q)\in I_2\cap I_3$, we see that $\overline q\in h(I_4)$, $h(\overline q)\in I_2$, $h(q)\in I_4$, and so, after \refcor{pingPongLinked}, we have $S^1=\overline {h(I_4)\cup I_2\cup h(J)\cup I_4}$. Observe also that since $h(I_3)\subset I_4$ and $h(I_2)\cap h(I_3)\neq\varnothing$, the condition $I_2\neq J$ gives $h(I_2)\subset I_4$ (again, this can be seen as a consequence of \refcor{pingPongLinked}). From this we see that $J\subset \opi{\overline q}{u_2}$. Now, by non-simplicity of $q$, we can take an element $f\in H_+(q)$ such that $f(I_2)$, which is a gap of $q$, is contained in $\opi{\beta}{u_2}$; since $f\in H_+(q)$, the gap $f(I_4)$ of $q$ is contained in $\opi{v_4}{\overline q}$. The gap $h(J)=\opi{h(\alpha)}{h(\beta)}$ contains the repelling fixed point of $f$; using also that $h(\beta)\in \overline {I_4}$ and $f(I_4)\subset \opi{v_4}{\overline q}$, we deduce that \begin{equation}\label{Eqn:parallelPingPong1}
       fh(J)\supset \overline{I_4}.
       \end{equation}
       On the other hand, let us see where $fh(I_4)$ is: the endpoint $h(v_4)$ of $h(I_4)$ belongs to $\overline{I_2}$, and the other endpoint $h(u_4)$ belongs to $I_4$; then $fh(u_4)\in \overline{f(I_2)}$ and $fh(v_4)\in f(I_4)$. We deduce that 
       \begin{equation}\label{Eqn:parallelPingPong2}
           fh(I_4)\supset \overline J.
       \end{equation}
       However, combining \eqref{Eqn:parallelPingPong1} and $\eqref{Eqn:parallelPingPong2}$, Combining \eqref{Eqn:parallelPingPong3} and \eqref{Eqn:parallelPingPong4} together, we see that $(fh)^2$ has repelling fixed points in both $J$ and $I_4$. Since these intervals are disjoint, this contradicts the assumption that $G$ is hyperbolic-like.

       Now, applying \refcor{pingPongLinked} to $q$ and $h(q)$, we see that the fact that $q\in h(I_2)$ and $h(\overline q)\in I_2$ implies $h(u_2)\in \lopi{h(\overline q)}{v_2}\subset \overline{I_2}$. Considering that $h(u_1)=u_3\in \overline {I_2}$, this gives the first desired inclusion $h(I_1\setminus \overline{I_2})=\opi{h(u_1)}{h(u_2)}\subset I_2$. For the other inclusion, observe that $I_3\setminus \overline {I_2}=\opi{v_2}{v_3}\in h(I_2)$, and thus $h^{-1}(I_3\setminus \overline{I_2})\subset I_2$, as desired.
   \end{proof}

    The inclusion relations from \refclm{pingPongLinked2} allow to conclude that $(U_p,U_q)$ is a proper ping-pong partition for $\Stab{G}{p}$ and $\Stab{G}{q}$. This ends the proof.
   \end{proof}
   \setcounter{claim}{0}

\appendix

\section{Configurations of fixed points}

Here we discuss the possible configurations of fixed points of two given distinct elements and their compositions in a hyperbolic-like group $G$. Recall that when $g\in G$ is a non-trivial element, we write $a(g)\in S^1$ for its attracting fixed point, and $r(g)$ for the repelling one.

\begin{prop}\label{Prop:configFix}
    Let $f$ and $g$ two non-commuting elements in a hyperbolic-like group $G$. Then, the relative positions of the attracting and repelling fixed points of the elements $f$, $g$, $fg$, and $gf$ are described by the 14 diagrams in Table \ref{Table:fixedPointsComposition}.
\end{prop}
\begin{proof}
Assume first that there is an open interval $I$ whose endpoints are the attracting fixed points $a(f)$ and $a(g)$ and such that the repelling fixed points $r(f)$ and $r(g)$ are not in $I$ (these are the cases described by the first row of Table \ref{Table:fixedPointsComposition}.
If so, we have $f(I)\cup g(I)\subset I$, and consequently $fg(I)\subset I$ and $gf(I)\subset I$. We deduce that the attracting fixed points $a(fg)$ and $a(gf)$ are in $I$. Moreover, the relation $g(a(fg))=a(gf)$ shows that $a(gf)$ is between $a(fg)$ and $a(g)$ in $I$. Similarly, the repelling fixed points are located considering the interval $J$ between the repelling fixed points of $f$ and $g$.

Assume now that we are not in the previous case. This corresponds to the other two rows in Table \ref{Table:fixedPointsComposition}. Note that the third row can be deduced from the second by replacing $f$ and $g$ with $f^{-1}$ and $g^{-1}$, respectively. Note also that the columns d and e in the second row can be deduced from the columns b and c, respectively, by exchanging the roles of $f$ and $g$. To treat the remaining cases 2a, 2b, and 2c, we study how the graphs of the homeomorphisms $f^{\pm 1}$ and $g^{\mp 1}$ intersect: indeed, fixed points of $fg$ correspond to points $x\in S^1$ such that $g(x)=f^{-1}(x)$; similarly, fixed points of $gf$ correspond to points $x\in S^1$ such that $f(x)=g^{-1}(x)$. The nature of these fixed points is deduced by looking at the way in which the graphs of the homeomorphisms cross: for instance, if $x$ is a fixed point of $gf$ and $f(y)<g^{-1}(y)$ for any point $y$ sufficiently close to $x$, and to the right of $x$, then $x$ is the attracting fixed point of $gf$.

Let us proceed to the proof. Recall that we have chosen the counterclockwise orientation of the circle. The cases in the second row correspond to the configuration
$a(f)<r(g)<a(g)<r(f)$. If this is the case, then the graph of $g$ crosses the graph of $f$ in the intervals $\opi{a(f)}{r(g)}$ and $\opi{a(g)}{r(f)}$, and there are no further crossings, since the element $f^{-1}g$ is hyperbolic.
We can assume that the graph of $g$ crosses the graph of $f^{-1}$ in the interval $\opi{a(f)}{r(f)}$, otherwise we exchange the roles of $f$ and $g$. This can happen in two distinct ways.
\begin{enumerate}
    \item The graph of $g$ crosses the graph of $f^{-1}$ in the interval $\opi{a(g)}{r(g)}$, and one checks that this situation is described by the diagram 2a in Table \ref{Table:fixedPointsComposition}. This situation is sketched in Figure \ref{Fig:crossings2a}.

    \item The graph of $g$ crosses the graph of $f^{-1}$ in the interval $\opi{r(g)}{a(g)}$. This leads to the diagrams 2b and 2c in Table \ref{Table:fixedPointsComposition}. Note that the two cases depend on the relative position of the first crossing of the graphs of $f$ and $g^{-1}$ and the second crossing of the graphs of $f^{-1}$ and $g$. These situations are sketched in Figure \ref{Fig:crossings2bc}.\qedhere

\end{enumerate}
\end{proof}

\begin{table}[ht]
\centering
\begin{tabular}{|r|c|c|c|c|c|}
    \hline
    & a & b & c & d & e\\ \hline 
    1 & \includegraphics[angle=135,origin=c,scale=.315]{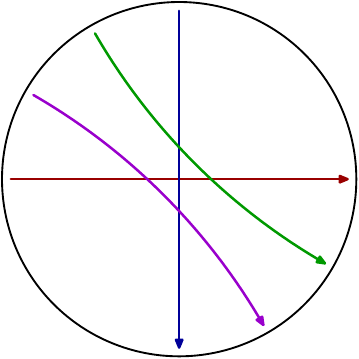} & \includegraphics[angle=45,origin=c,scale=.315]{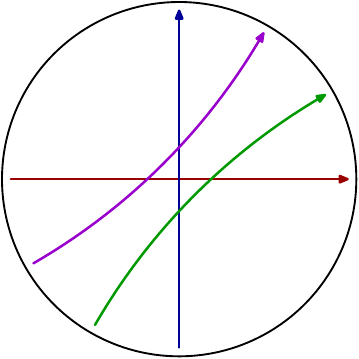} &
    \includegraphics[angle=135,origin=c,scale=.315]{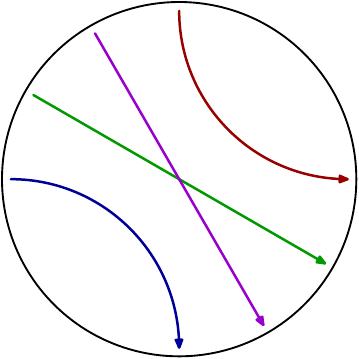} & \includegraphics[angle=-45,origin=c,scale=.315]{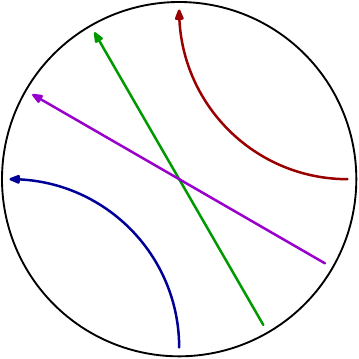} &\\ \hline
    2 & \includegraphics[angle=-45,origin=c,scale=.315]{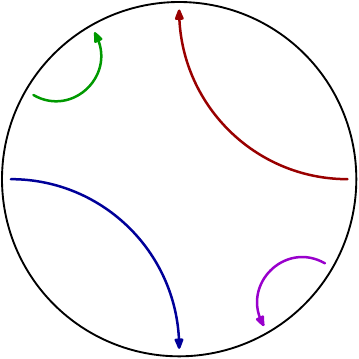} & \includegraphics[angle=-45,origin=c,scale=.315]{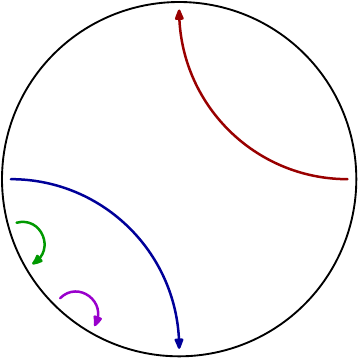} & \includegraphics[angle=-45,origin=c,scale=.315]{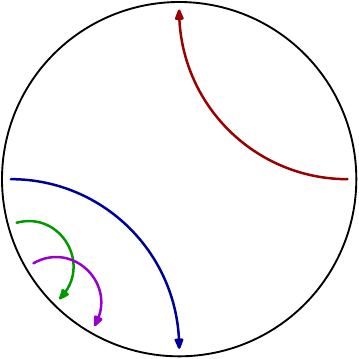}& \includegraphics[angle=-45,origin=c,scale=.315]{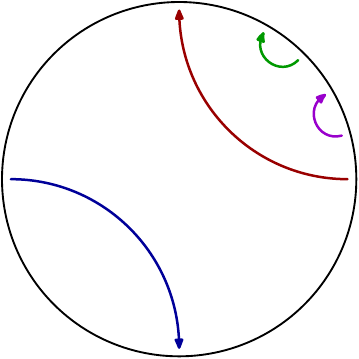} & \includegraphics[angle=-45,origin=c,scale=.315]{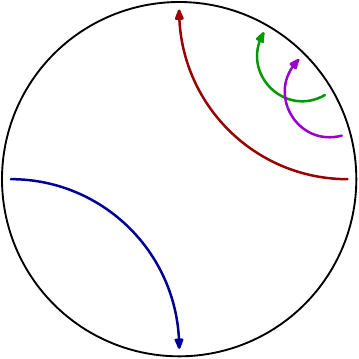}\\ \hline
    3 &\includegraphics[angle=-45,origin=c,scale=.315]{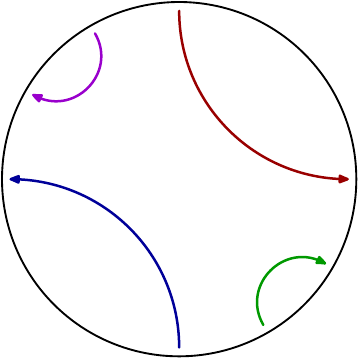} & \includegraphics[angle=-45,origin=c,scale=.315]{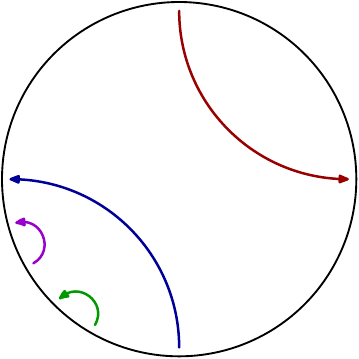} & \includegraphics[angle=-45,origin=c,scale=.315]{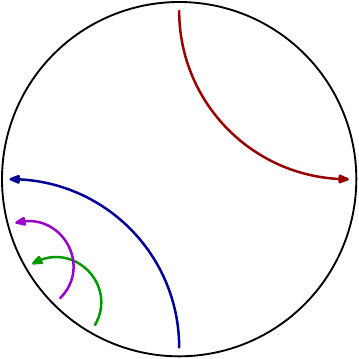}& \includegraphics[angle=-45,origin=c,scale=.315]{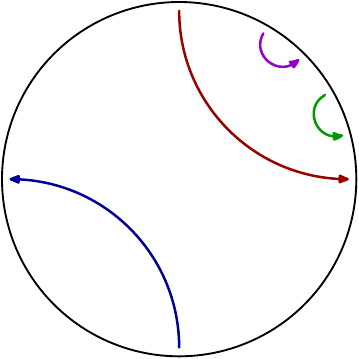} & \includegraphics[angle=-45,origin=c,scale=.315]{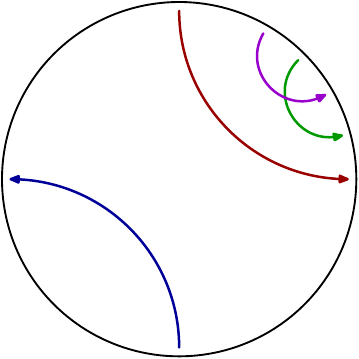}\\ \hline
\end{tabular}

\caption{Possible configurations of attracting and repelling fixed points of two elements $f$ (red), $g$ (blue), and their compositions $fg$ (green), $gf$ (purple). Arrows go from the repelling fixed point to the attracting fixed point of the corresponding element.}\label{Table:fixedPointsComposition}
\end{table}

\begin{figure}[ht]
    \centering
    \includegraphics[width=0.5\linewidth]{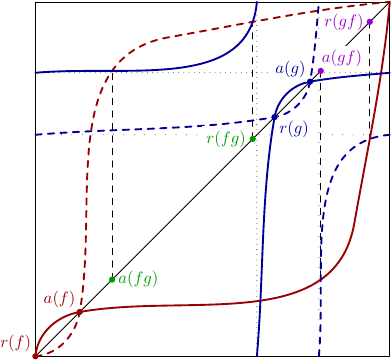}
    \caption{Crossings of graphs of $f^{\pm 1}$ (red) and $g^{\mp1}$ (blue) corresponding to the diagram 2a in Table~\ref{Table:fixedPointsComposition}. The dashed graphs are those of the inverses $f^{-1}$ and $g^{-1}$.}
    \label{Fig:crossings2a}
    \end{figure}

    \begin{figure}[ht]
    \centering
    \includegraphics[width=0.45\linewidth]{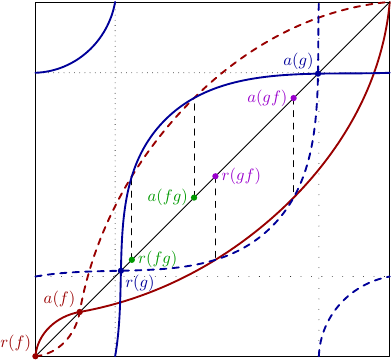}\hspace{2em}\includegraphics[width=0.45\linewidth]{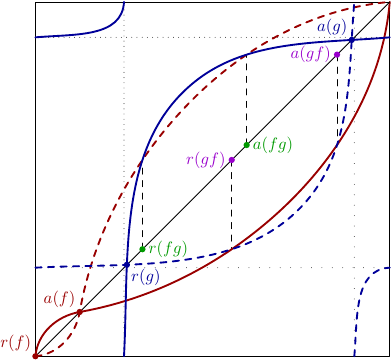}
    \caption{Crossings of graphs of $f^{\pm 1}$ (red) and $g^{\mp1}$ (blue) corresponding to the diagrams 2b (left) and 2c (right) in Table \ref{Table:fixedPointsComposition}. The dashed graphs are those of the inverses $f^{-1}$ and $g^{-1}$.}
    \label{Fig:crossings2bc}
    \end{figure}

We use the previous result to analyze the configuration of fixed points of commutators of elements from different stabilizers. Recall that we write $[f,g]=fgf^{-1}g^{-1}$ for any $f,g\in \Homeop(S^1)$, and that given a point $p\in S^1$, we write $H_+(p)=\{h\in G: a(h)=p\}$. When $H_+(p)$ is non-empty, we write $\overline p$ for the other fixed point of any $h\in H_+(p)$ (this does not depend on the choice of such $h$).

\begin{prop}\label{Prop:fixCommutator}
    Let $G$ be a hyperbolic-like group. Suppose that $p$ and $q$ are points with non-trivial stabilizers such that $p<q<\overline{p}<\overline{q}<p$.
    Then, for any $f\in H_+(q)$ and $h\in H_+(p)$, one of the following occurs. (See Figure \ref{Fig:fixCommutator}.)
    \begin{enumerate}[label=(\arabic*)]
        \item\label{Itm:geometric} \emph{Geometric configuration}. We have
    \[ \begin{array}{cccccccccccc}
        &p&<&h^{-1}(q)&<&r([f^{-1},h^{-1}])&<&a([f^{-1},h^{-1}])&<&f^{-1}(p)&\\
        <&q&<&f^{-1}(\overline p)&<&r([h,f^{-1}])&<&a([h,f^{-1}])&<&h(q)&\\
        <&\overline p&<&h(\overline q)&<&r([f,h])&<&a([f,h])&<&f(\overline p)&\\
        <&\overline q&<& f(p)&<& r([h^{-1},f])&<&a([h^{-1},f])&<&h^{-1}(q)&<&p.
    \end{array}\]

    \item \emph{Non-geometric configurations}.
    \begin{enumerate}[label=(2.\Roman*)]
        \item\label{Itm:nonGeometricI}  \emph{Fixed points of all commutators in $\opi{p}{q}$.} We have    
    \[
    p<f^{-1}(p)<a([f^{-1},h^{-1}])<r([f^{-1},h^{-1}])<h^{-1}(q)<q,
    \]
    in which case, we also have
    \begin{align*}
    &\Fix([h^{-1},f])=f(\Fix([f^{-1},h^{-1}]))\subset \opi{p}{fh^{-1}(q)},
    \\
    &\Fix([h,f^{-1}])=h(\Fix([f^{-1},h^{-1}]))\subset \opi{hf^{-1}(p)}{q},
    \\
    &\Fix([f,h])=f(\Fix([h,f^{-1}]))=h(\Fix([h^{-1},f]))\subset \opi{fhf^{-1}(p)}{hfh^{-1}(q)}.
    \end{align*}
    
    \item\label{Itm:nonGeometricII} \emph{Fixed points of all commutators in $\opi{q}{\overline p}$.} We have
    \[
    q<h(q)<a([h,f^{-1}])<r([h,f^{-1}])<f^{-1}(\overline p)<\overline p,
    \]
    in which case, we also have
     \begin{align*}
    &\Fix([f,h])=f(\Fix([h,f^{-1}]))\subset \opi{fh(q)}{\overline p},
    \\
    &\Fix([f^{-1},h^{-1}])=h^{-1}(\Fix([h,f^{-1}]))\subset \opi{q}{h^{-1}f^{-1}(\overline p)},
    \\
    &\Fix([h^{-1},f])=h^{-1}(\Fix([f,h]))=f(\Fix([f^{-1},h^{-1}]))\subset \opi{h^{-1}fh(q)}{fh^{-1}f^{-1}(\overline p)}.
    \end{align*}
    \item\label{Itm:nonGeometricIII} \emph{Fixed points of all commutators in $\opi{\overline p}{\overline q}$.} We have
    \[
    \overline p<f(\overline p)<a([f,h])<r([f,h])<h(\overline q)<\overline q,
    \]
    in which case, we also have
    \begin{align*}
    &\Fix([h,f^{-1}])=f^{-1}(\Fix([f,h]))\subset \opi{\overline p}{f^{-1}h(\overline q)},
    \\
    &\Fix([h^{-1},f])=h^{-1}(\Fix([f,h]))\subset \opi{h^{-1}f(\overline p)}{\overline q},
    \\
    &\Fix([f^{-1},h^{-1}])=h^{-1}(\Fix([h,f^{-1}]))=f^{-1}(\Fix([h^{-1},f]))\subset \opi{f^{-1}h^{-1}f(\overline p)}{h^{-1}f^{-1}h(\overline q)}.
    \end{align*}
    \item\label{Itm:nonGeometricIV} \emph{Fixed points of all commutators in $\opi{\overline q}{p}$.} We have
    \[
    \overline q<h^{-1}(\overline q)<a([h^{-1},f])<r([h^{-1},f])<f(p)<p,
    \]
    in which case, we also have
     \begin{align*}
    &\Fix([f,h])=h(\Fix([h^{-1},f]))\subset \opi{\overline q}{hf(p)},
    \\
    &\Fix([f^{-1},h^{-1}])=f^{-1}(\Fix([h^{-1},f]))\subset \opi{f^{-1}h^{-1}(\overline q)}{p},
    \\
    &\Fix([h,f^{-1}])=f^{-1}(\Fix([f,h]))=h(\Fix([f^{-1},h^{-1}]))\subset \opi{hf^{-1}h^{-1}(\overline q)}{f^{-1}hf(p)}.
    \end{align*}
    \end{enumerate}
    \end{enumerate}
\end{prop}

\begin{figure}[ht]
        \centering
        \includegraphics[width=0.45\linewidth]{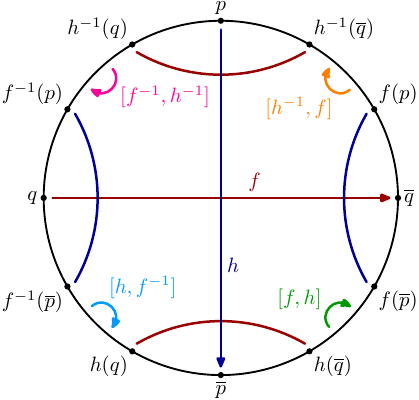}\quad \includegraphics[width=0.45\linewidth]{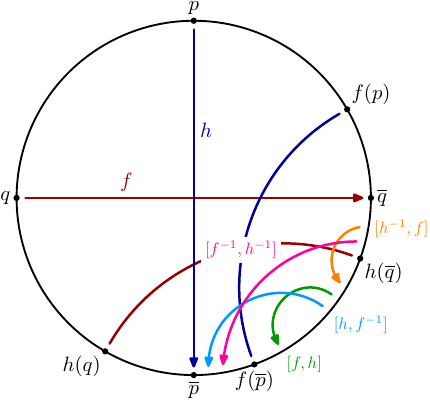}
        \caption{Configurations of fixed points of commutators in \refprop{fixCommutator}.
        Left: geometric configuration. Right:  fixed points in $\opi{\overline p}{\overline q}$.}
        \label{Fig:fixCommutator}
    \end{figure}

\begin{proof}
    Let us start with some remarks. First, the four commutators $[f,h]$, $[h,f^{-1}]$, $[h^{-1},f]$, and $[f^{-1},h^{-1}]$ are related by conjugacy in the following way:
    \[\begin{tikzcd}
    \left [f^{-1},h^{-1}\right ] \arrow[r,"h"] \arrow[d,"f"] & \left[h,f^{-1}\right] \arrow[d,"f"] \\ \left[h^{-1},f\right] \arrow[r,"h"] & \left[f,h\right]
    \end{tikzcd}\]
    (with this diagram we mean that the element on the arrow conjugates the commutator at its source to the commutator at its target: for example, in the first row we have the relation $h[f^{-1},h^{-1}]h^{-1}=[h,f^{-1}]$). Consequently, if $\Fix([f^{-1},h^{-1}])$ is contained in some interval $I$, then $\Fix([h,f^{-1}])=h(\Fix([f^{-1},h^{-1}]))\subset h(I)$, \textit{etc.} This is used to deduce where the fixed points of the other commutators are located, especially in the non-geometric configurations. Second, we work out a preliminary reduction of cases: we check with the help of Table \ref{Table:fixedPointsComposition} which configurations of fixed points of $[f,h]$ are compatible with both factorizations
    \begin{equation}\label{Eqn:fixCommutator1}
        [f,h]=f\cdot (hfh^{-1})^{-1}=(fhf^{-1})\cdot h^{-1}.
    \end{equation}
    For both factorizations, we are in the cases described in row 3 in Table \ref{Table:fixedPointsComposition}; as we only consider one composition, there is no difference here between diagrams 3b and 3c, and between diagrams 3d and 3e. This leads to $3^2=9$ total possibilities, and a case-by-case analysis shows that only five are compatible with the two factorizations in \eqref{Eqn:fixCommutator1}; we list them in Table \ref{Table:fixedPointsCommutators}.

\begin{table}[ht!]
\centering
\begin{tabular}{|r|c|c|L|}
    \hline
    & $f\cdot (hfh^{-1})^{-1}$ & $(fhf^{-1})\cdot h^{-1}$ & $\Fix([f,h])$ in\\ \hline 
    1 & 3d/e & 3b/c & \vspace{.1em} $\opi{p}{q}$
    \\
    & \includegraphics[scale=0.7]{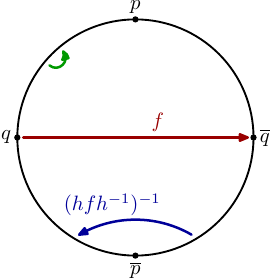} & \includegraphics[scale=0.7]{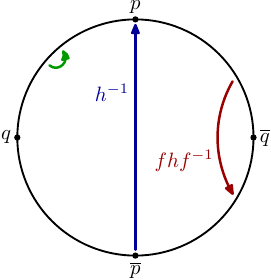}& \\ \hline
    2 & 3b/c & 3b/c & \vspace{.1em} $\opi{h(q)}{\overline p}$
    \\ [.5em]
    & \includegraphics[scale=0.7]{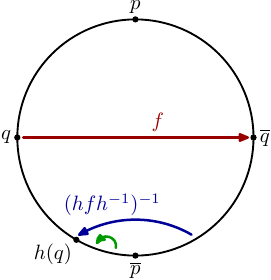} & \includegraphics[scale=0.7]{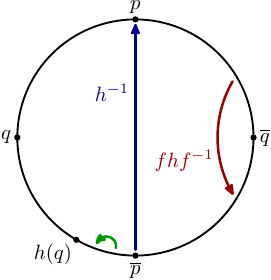}& \\ \hline
    3 & 3a & 3a & \vspace{.1em} $\opi{h(\overline q)}{f(\overline p)}$   
   \mbox{\footnotesize{(only when $h(\overline q)\in \opi{\overline p}{f(\overline p)}$)}}
    \\
    & \includegraphics[scale=0.7]{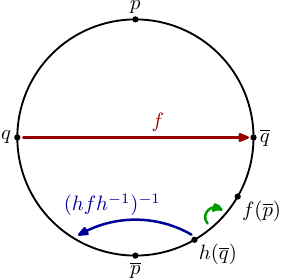} & \includegraphics[scale=0.7]{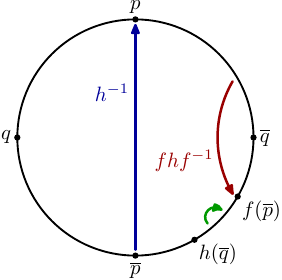}&    
    \\ \hline
    4 & 3b/c & 3d/e & \vspace{.1em} $\opi{f(\overline{p})}{h(\overline q)}$
    \mbox{\footnotesize{(only when $f(\overline p)\in \opi{\overline p}{h(\overline q)}$)}}
     \\ [.5em]
    &\includegraphics[scale=0.7]{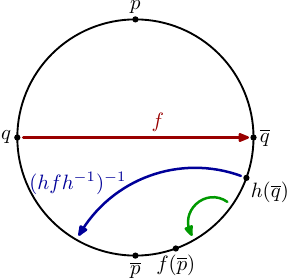}&\includegraphics[scale=0.7]{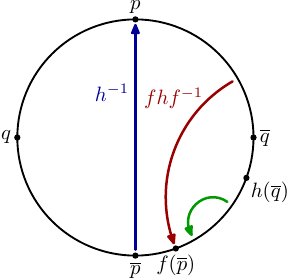}  & \\ \hline
    5 &  3d/e & 3d/e & \vspace{.1em} $\opi{\overline{q}}{f(p)}$
    \\ [.5em]
    & \includegraphics[scale=0.7]{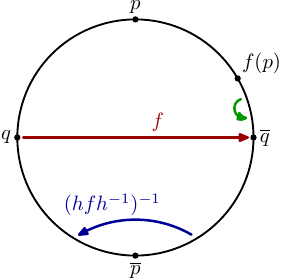}&\includegraphics[scale=0.7]{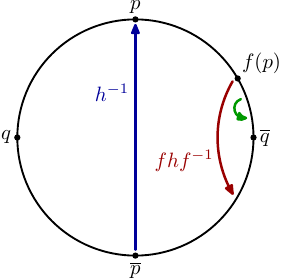} & \\ \hline
\end{tabular}

\caption{Proof of \refprop{fixCommutator}: first analysis of configurations of fixed points of a commutator $[f,h]$, where $f\in H_+(q)$ and $h\in H_+(p)$, in the linked situation $p<q<\overline p<\overline q$.}\label{Table:fixedPointsCommutators}
\end{table}

Now, consider first the case described by the fourth row in Table \ref{Table:fixedPointsCommutators}: in this case we have
\[
    \overline p<f(\overline p)<a([f,h])<r([f,h])<h(\overline q)<\overline q,
\]
and this corresponds to the non-geometric case \ref{Itm:nonGeometricIII}.

\begin{figure}[ht]
        \centering
        \includegraphics[width=0.5\linewidth]{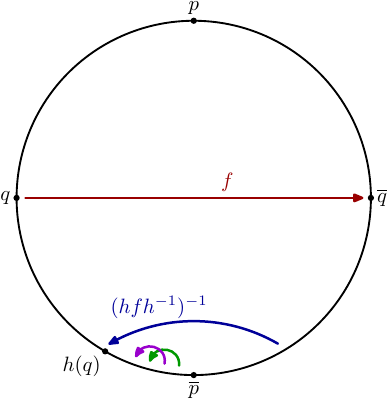}
        \caption{Proof of \refprop{fixCommutator}: configuration of fixed points for $[f,h]$ (green) and $[h,f^{-1}]$ (purple) assuming $\Fix([f,h])\subset \opi{h(q)}{\overline p}$.}
        \label{Fig:fixedPointsCommutators1}
    \end{figure}

    Next, let us focus on the second row in Table \ref{Table:fixedPointsCommutators}:
    assume that $\Fix([f,h])\subset \opi{h(q)}{\overline p}$. We reproduce in Figure \ref{Fig:fixedPointsCommutators1} the diagram 3b/c for the factorization $[f,h]=f\cdot (hfh^{-1})^{-1}$ that describes this case, including also the fixed points of the reverse composition $(hfh^{-1})^{-1}\cdot f=[h,f^{-1}]$: we must have
    \[
    q < h(q) < a([h,f^{-1}])<r([h,f^{-1}])<\overline p.
    \]
    When considering the commutator $[h,f^{-1}]$ instead of $[f,h]$, this situation corresponds to the one described in row 4 in Table \ref{Table:fixedPointsCommutators}. Therefore, in this case, we have
    \[q < h(q) < a([h,f^{-1}])<r([h,f^{-1}])<f^{-1}(\overline p)<\overline p,\]
    and this corresponds to the non-geometric case \ref{Itm:nonGeometricII} in the statement. Arguing similarly with the configuration of fixed points in row 5 (and considering the commutator $[h^{-1},f]$ instead of $[h,f^{-1}]$), we conclude that this case leads to the non-geometric configuration \ref{Itm:nonGeometricIV}.
    
    Next, consider the case described in row 3 in Table \ref{Table:fixedPointsCommutators}, and as before, we locate the fixed points of the other commutators $[h,f^{-1}]$ and $[h^{-1},f]$. See Figure \ref{Fig:fixedPointsCommutators3}. These also correspond to the case 3a/3a, and we conclude that the fixed points of the commutators $[f,h]$, $[h,f^{-1}]$, and $[h^{-1},f]$, are as described in the geometric configuration \ref{Itm:geometric}. Repeating this argument with one of the commutators $[h,f^{-1}]$ or $[h^{-1},f]$, we deduce that the fixed points of $[f^{-1},h^{-1}]$ also respect the geometric configuration \ref{Itm:geometric}.

\begin{figure}[ht]
        \centering
        \includegraphics[width=0.4\linewidth]{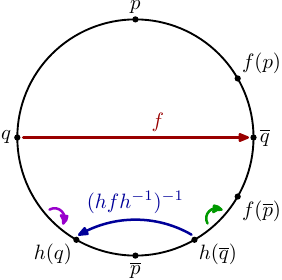}\quad\includegraphics[width=0.4\linewidth]{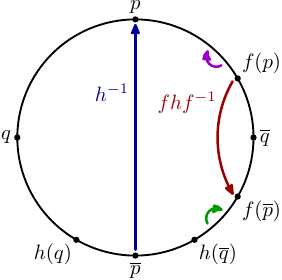}
        \caption{Proof of \refprop{fixCommutator} in the case $\Fix([f,h])\subset \opi{h(\overline q)}{f(\overline p)}$. Left: configuration of fixed points for $[f,h]$ (green) and $[h,f^{-1}]$ (purple). Right: configuration of fixed points for $[f,h]$ (green) and $[h^{-1},f]$ (purple).}
        \label{Fig:fixedPointsCommutators3}
    \end{figure}
    
    To conclude, consider the case $\Fix([f,h])\subset \opi{p}{q}$, described in row 1 of Table \ref{Table:fixedPointsCommutators}. When this happens, we also have $\Fix([h,f^{-1}])\subset \opi{p}{q}$. The arguments given so far, working with the commutator $[h,f^{-1}]$ instead of $[f,h]$, allow to conclude that this corresponds to the non-geometric situation \ref{Itm:nonGeometricI}.\qedhere

    \begin{figure}[ht]
        \centering
        \includegraphics[width=0.4\linewidth]{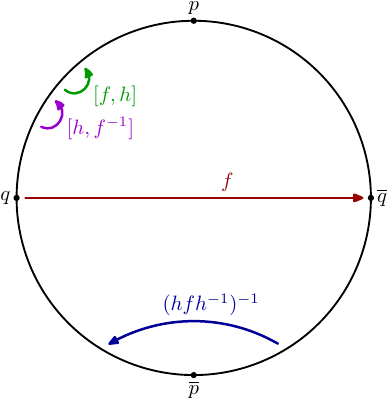}
        \caption{Proof of \refprop{fixCommutator} in the case $\Fix([f,h])\subset \opi{p}{q}$. Configuration of fixed points for $[f,h]$ (green) and $[h,f^{-1}]$ (purple).}
        \label{Fig:fixedPointsCommutators4}
    \end{figure}

\end{proof}

\newpage 

\section*{Acknowledgments}
The authors would like to thank João Carnevale and Christian Bonatti for several conversations that were at the origin of this work.
The first author was supported by the National Research Foundation of Korea Grant funded by the Korean Government (RS-2022-NR072395).
The second author was partially supported by the projects ANR Gromeov (ANR19-CE40-0007), AnoDyn (ANR-24-CE40-5065-01), and by the EIPHI Graduate School (ANR-17-EURE-0002).

\bibliographystyle{alpha}
\bibliography{biblio.bib}

\end{document}